%% file: DiscreteSystolic.tex
\documentclass[10pt]{amsart}

\title{Combinatorial systolic inequalities}

\author{Ryan Kowalick}
\address{422 E 17th Ave. \\
                 Columbus, Ohio 43201}
\email{rkowalick@gmail.com}

\author{Jean-Fran\c{c}ois Lafont}
\address{Department of Mathematics\\
                 Ohio State University\\
                 Columbus, Ohio 43210}
\email{jlafont@math.ohio-state.edu}

\author{Barry Minemyer}
\address{Department of Mathematics\\
                 Ohio State University\\
                 Columbus, Ohio 43210}
\email{minemyer.1@osu.edu}

\date{\today}

\newtheorem{thm}{Theorem}
\newtheorem*{theorem}{Main Theorem}
\newtheorem{lem}[thm]{Lemma}
\newtheorem*{cor}{Corollary}

\theoremstyle{definition}
\newtheorem*{rem}{Remark}
\newtheorem{claim}{\bf Claim}

\numberwithin{equation}{section}

\newcommand{\todo}[1]{\vspace{5mm}\par \noindent
\framebox{\begin{minipage}[c]{0.95 \textwidth} \tt #1
\end{minipage}} \vspace{5mm} \par}

\renewcommand{\a}{\alpha}
\renewcommand{\b}{\beta}
\renewcommand{\d}{\delta}

\newcommand{\e}{\varepsilon}
\newcommand{\g}{\gamma}

\newcommand{\s}{\sigma}

\renewcommand{\t}{\tau}
\renewcommand{\O}{\Omega}

\newcommand{\T}{\mathcal T}

\def\Pb{\ifmmode{\Bbb P}\else{$\Bbb P$}\fi}
\def\Z{\ifmmode{\Bbb Z}\else{$\Bbb Z$}\fi}
\def\Q{\ifmmode{\Bbb Q}\else{$\Bbb Q$}\fi}
\def\C{\ifmmode{\Bbb C}\else{$\Bbb C$}\fi}
\def\R{\ifmmode{\Bbb R}\else{$\Bbb R$}\fi}
\def\H{\ifmmode{\Bbb H}\else{$\Bbb H$}\fi}
\def\S{\ifmmode{S^2}\else{$S^2$}\fi}
\def\U{\mathcal{U}}
\def\V{\mathcal{V}}
\def\N{\mathcal{N}}
\def\ds{\displaystyle}
\def\kp{\kappa^\prime}
\def\vxi{v_i}
\def\vxj{v_j}
\def\vyi{w_i}
\def\vyj{w_j}
\DeclareMathOperator{\ind}{ind}
\DeclareMathOperator{\alt}{alt}

\def\B{\mathcal B}

\def\E{\mathbb{E}}

\def\dist{\operatorname{dist}}

\def\diam{\operatorname{diam}}

\def\S{\mathcal S}

\def\sys{\operatorname{Sys}}

\def\vol{\operatorname{Vol}}

\usepackage{graphicx} 
\usepackage{amsmath,amssymb,amsthm,amsfonts,amscd,flafter,epsf}
\usepackage{verbatim}
\usepackage{graphics}
\usepackage{epsfig}
\usepackage{color}
\usepackage{tikz} 
\usetikzlibrary{calc, intersections} 
\input xy
\xyoption{all}

\begin{document}

\begin{abstract}
We establish combinatorial versions of various classical systolic inequalities. For a smooth triangulation
of a closed smooth manifold, the minimal number of edges in a homotopically 
non-trivial loop contained in the $1$-skeleton gives an integer called the combinatorial
systole. The number of top-dimensional simplices in the triangulation gives another
integer called the combinatorial volume. We show that a class of smooth manifolds satisfies
a systolic inequality for all Riemannian metrics if and only if it satisfies a corresponding 
combinatorial systolic inequality for all smooth triangulations. Along the way, we show
that any closed Riemannian manifold has a smooth triangulation which ``remembers''
the geometry of the Riemannian metric, and conversely, that every smooth triangulation gives rise 
to Riemannian metrics which encode the combinatorics of the triangulation. 
We give a few applications of these results.

\end{abstract}

\maketitle

\setcounter{secnumdepth}{1}

\setcounter{section}{0}

\section{\bf Introduction}

For a closed Riemannian manifold $(M, g)$, the {\it systole} is the minimal length of a homotopically non-trivial loop, denoted
$\sys_g(M)$, while the volume of $(M, g)$ is denoted $\vol_g(M)$. Systolic inequalities are expressions which relate the systole with
other geometric quantities, typically the volume. In this paper, we are interested in {\it combinatorial} versions of the systolic inequalities.

We view smooth triangulations of a manifold $M$ as a combinatorial
model for $M$. For such a triangulation $(M, \T)$, we define the combinatorial systole $\sys_\T(M)$ to be the minimal number of
edges for a combinatorial loop in the $1$-skeleton of $\T$ which is homotopically non-trivial in $M$. The discrete volume 
$\vol_\T(M)$ is just the number of top-dimensional simplices in the triangulation $\T$. The main goal of this paper is to establish
the following:

\vskip 10pt

\begin{theorem}\label{maintheorem}
Let $\mathcal M$ be a class of closed smooth $n$-manifolds. Then the following two statements are equivalent:
\begin{enumerate}
\item for every Riemannian metric $(M, g)$ on a manifold $M\in \mathcal M$, we have
$$\sys_g(M) \leq C \sqrt[n]{\vol_g(M)},$$
where $C$ is a constant which depends solely on the class $\mathcal M$.
\item for every smooth triangulation $(M, \T)$ of a manifold $M\in \mathcal M$, we have
$$\sys_\T(M) \leq C^\prime \sqrt[n]{\vol_\T(M)},$$
where $C^\prime$ is a constant which depends solely on the class $\mathcal M$.
\end{enumerate}
\end{theorem}

In \cite{Gromov1} Gromov proved that the class of closed smooth \emph{essential} Riemannian manifolds satisfies the above Riemannian systolic inequality.  So an immediate consequence is the following.

\begin{cor}\label{cor:essential}
Let $\mathcal M$ denote the class of closed smooth essential $n$-manifolds.  Then for every smooth triangulation $(M,\T)$ of a manifold $M \in \mathcal M$, we have
$$\sys_\T(M) \leq C \sqrt[n]{\vol_\T(M)},$$
where $C$ is a constant which depends solely on the class $\mathcal M$.
\end{cor}

In the process of proving our {\bf Main Theorem}, we establish a number of auxiliary results which might also be of 
independent interest. After some preliminaries in Section \ref{section:background}, we show:

\begin{thm}[Encoding a Riemannian metric]\label{thm:discrete-to-riemannian}
There exists a constant $\delta_n$ depending solely on the dimension $n$, with the property that for any
closed Riemannian manifold $(M, g)$, there exists a smooth triangulation $\T$ with the property that 
$$\frac{\sup_{e\subset \T}\{l_g(e)\}}{\inf_{\sigma \subset \T}\{\sqrt[n]{\vol_g(\sigma)}\}} \leq \delta _n,$$
where the volume of the top-dimensional simplices $\sigma$, and the lengths of the edges $e$, 
are measured in the ambient $g$-metric.
\end{thm}

Roughly speaking, the triangulation $\T$ produced in the theorem has no simplices that are ``long and thin''
(as measured in the Riemannian metric $g$). 
Moreover, Theorem \ref{thm:discrete-to-riemannian} still holds for a possibly different constant $\d_{n,k}$ when considering the collection of $k$-dimensional simplices $\s$ (and replacing $\sqrt[n]{\text{Vol}_g(\s)}$ with the $k^{th}$ root of the $k$-dimensional volume of $\s$).
Theorem \ref{thm:discrete-to-riemannian} is established in Section \ref{section:discrete-to-riemannian}.
The basic idea is as follows. In \cite{Whitney}, Whitney proved that every closed smooth manifold $M^n$ supports a triangulation.
Note that in dimensions $\geq 4$, the corresponding statement is {\it false} for topological manifolds (work of Freedman and Casson in dimension $=4$ \cite{Fr}, and of Manolescu \cite{Ma} in dimensions $\geq 5$), see Section \ref{section:concluding-remarks}.
The method Whitney used was to first smoothly embed $M^n$ into $\mathbb R^{2n+1}$, and equip the latter with 
a sufficiently fine cubulation. Then one perturbs the embedding to be transverse to the $(n+1)$-cubes in the cubulation -- the
intersection will then give a collection of points. One then uses these points as the vertex set of a certain piecewise affine 
(polyhedral) complex in $\mathbb R^{2n+1}$. If the cubulation is chosen fine enough, this complex lies in a small normal 
neighborhood of $M^n$, and one can subdivide to get a simplicial complex, then project down onto $M^n$. 
Whitney then argues that this projection provides a smooth triangulation of $M$.

Now the proof of Theorem \ref{thm:discrete-to-riemannian} also uses Whitney's procedure, but rather than starting
from a smooth embedding into $\mathbb R^{2n+1}$, we want to start with an embedding that ``remembers'' the Riemannian
structure on $(M,g)$. A natural choice to use is Nash's isometric embedding. We then follow through Whitney's arguments,
and check that the resulting triangulation has the desired property. This is done in Section \ref{section:whitneyarg}.

\begin{thm}[Encoding a triangulation]\label{thm:riemannian-to-discrete}

There exists a constant $\kappa_n$ depending solely on the dimension $n$, with the property that for any
smooth triangulation $(M, \T)$ of a smooth compact manifold $M$, and for any $\e > 0$, there exists a Riemannian metric 
$g$ on $M$ which satisfies the following:

\begin{enumerate}

\item $|\text{Vol}_{g}(M) - Vol_\T(M)| < \e$

\item If $\g$ is a closed path on $M$, then there exists a closed edge loop $p$, freely homotopic to $\g$, so that 
$$ l_{\T}(p) \leq  \kappa_n l_{g}(\g).$$ 

\end{enumerate}

\end{thm}

The idea behind the proof is to put a piecewise Euclidean metric on $M$, by making each
$n$-dimensional simplex in the triangulation $\T$ isometric to a Euclidean simplex with all edges of equal length, and of
volume $=1$. This metric has singularities along the codimension two strata, which can be inductively smoothed out. 
This gives a metric $g$ satisfying property (1). 
For property (2), one can easily reduce to the case that $\g$ is a $g$-geodesic which is not null-homotopic.  
From there, we remove the sections of $\g$ near the codimension 2 skeleton and, in a Lipschitz manner, replace them with geodesic segments in the singular metric. This results in a loop of roughly comparable length in the singular metric, and
property (2) is easy to establish for the singular metrics.
The details of this argument can be found in Section \ref{section:riemannian-to-discrete}.

In the proof of Theorem \ref{thm:riemannian-to-discrete} we only use the assumption that the triangulation $\T$ is compatible with the smooth structure on $M$ in one spot:  when using a smooth partition of unity to patch together locally defined metrics in the construction of the Riemannian metric $g$.  We need the metric $g$ to be smooth in order to use Theorem \ref{thm:riemannian-to-discrete} to prove one implication in our {\bf Main Theorem}.  But if all one requires is a $C^0$-Riemannian metric satisfying the two statements in  Theorem \ref{thm:riemannian-to-discrete}, then the assumption can be weakened to $\T$ being a \emph{piecewise linear} triangulation of $M$.  Our technique of proof does not extend to continuous triangulations, unfortunately.  See Section \ref{section:concluding-remarks} for a further discussion.
As a final remark about Theorem \ref{thm:riemannian-to-discrete}, we observe that by simply scaling the metric $g$, one may obtain equality in property (1) above at the cost of slightly altering the Lipschitz constant $\kappa_n$ in property (2).

\vskip 10pt

Using these two theorems, the proof of the {\bf Main Theorem} is easy.

\begin{proof}[Proof of Main Theorem]

{\bf ($\Rightarrow$)} Assume you have a class $\mathcal M$ of smooth $n$-manifolds satisfying condition (1) of the theorem, i.e.
satisfying a Riemannian systolic inequality. Let $\T$ be a smooth triangulation of a manifold $M\in \mathcal M$ lying within the class,
and $\epsilon >0$ an arbitrary positive constant.
Let $g$ be the Riemannian metric on $M$ whose existence is provided by our Theorem \ref{thm:riemannian-to-discrete}, 
$\gamma$ the closed $g$-geodesic whose length realizes the Riemannian systole of $(M,g)$, and $p$ the edge path freely
homotopic to $\gamma$ given by Theorem  \ref{thm:riemannian-to-discrete}. Then we 
have the sequence of inequalities:
\begin{align*}
\sys_\T(M) &\leq  l_{\T}(p) \leq  \kappa_n l_{g}(\g) = \kappa_n \cdot \sys _g(M)\\
& \leq \kappa_n \cdot C \sqrt[n]{\vol_g(M)} \\
& \leq (\kappa_n\cdot C) \sqrt[n]{\vol_{\T}(M) +\epsilon} 
\end{align*}
Letting $\epsilon$ tend to zero, we see that the class $\mathcal M$ satisfies condition (2) of the theorem (i.e. satisfies a combinatorial systolic 
inequality), with constant $C^\prime = \kappa_n \cdot C$.

\vskip 5pt

\noindent {\bf ($\Leftarrow$)} Conversely, let us assume that you have a class $\mathcal M$ of smooth $n$-manifolds satisfying condition 
(2) of the theorem, i.e. satisfying a combinatorial systolic inequality. Let $g$ be an arbitrary Riemannian metric on one of the manifolds 
$M\in \mathcal M$ lying within the class.  Let $\T$ be the smooth triangulation of $M$ obtained by applying our Theorem 
\ref{thm:discrete-to-riemannian}. We denote by $E$ the supremum of the $g$-lengths of edges in $\T$, and by $v$ the infimum of the
volume of top dimensional simplices in $\T$. So by Theorem \ref{thm:discrete-to-riemannian}, we have that $\frac{E}{v^{1/n}} \leq \delta_n$.
Let $p$ be an edge path in the triangulation $\T$ which realizes the combinatorial systole. 
Then we have the series of inequalities:
\begin{align*}
\sys_g(M) &\leq  l_{g}(p) \leq E\cdot l_\T(p) = E\cdot \sys_\T(M)   \\
& \leq C^\prime \cdot E\cdot \sqrt[n]{\vol_\T(M)} \leq C^\prime \cdot E\cdot \sqrt[n]{\frac{\vol_g(M)}{v}} = \delta_n C^\prime \cdot \sqrt[n]{\vol_g(M)}
\end{align*}
Thus, we see that the class $\mathcal M$ satisfies condition (1) of the theorem (i.e. satisfies a Riemannian systolic inequality), with 
constant $C= \delta_n \cdot C^\prime$. This concludes the proof of our {\bf Main Theorem}.
\end{proof}

\vskip 10pt

After the proof of Theorem \ref{thm:riemannian-to-discrete}, we discuss some applications of our {\bf Main Theorem}  in Section \ref{section:applications}.  Our paper concludes with a discussion about some open problems in Section \ref{section:concluding-remarks}, and an Appendix listing some general topology results due to Whitney \cite{Whitney} which are used in Section \ref{section:whitneyarg}.

\begin{rem} Most of Sections \ref{section:discrete-to-riemannian}, \ref{section:whitneyarg}, and \ref{section:applications} are contained in the Ph. D. Thesis of Ryan Kowalick \cite{Kowalick}. Similar results were independently obtained by 
de Verdi\`{e}re, Hubard, and de Mesmay \cite{VHM}. Their results are focused on the $2$-dimensional closed surfaces case (and includes other applications), while in the present paper we are able to deal with all dimensions.
\end{rem}

\subsection*{Acknowledgments} The authors would like to thank Dylan Thurston for some helpful comments.  The work of the second author was partially supported by the NSF, under grants DMS-1207782 and DMS-1510640. 


\section{Background material and notation}\label{section:background}

Suppose $X$ is a metric space, with $S$, $T \subset X$. We define the \emph{distance between $S$ and $T$} by 
\begin{equation*}
  \dist(S,T) = \inf\{ d(s,t) \,:\, s \in S, t \in T\}
\end{equation*}
and we note that this definition remains unchanged in the event that either $S$ or $T$ consists of a single point.
Also, we define the \emph{ $r$-neighborhood of $S$ in $X$}, denoted $U_r(S)$ to be
\begin{equation*}
  U_r(S) = \{ x \in X \,:\, \dist(x, S) < r \}.
\end{equation*}

In this paper all manifolds, (Riemannian) metrics, and (simplicial) triangulations are assumed to be smooth.  A \emph{triangulated manifold} is a tuple $(M, \T)$ where $M$ is a manifold and $\T$ is a triangulation.  When there is the possibility of confusion, we will denote the triangulation of a manifold $M$ by $\T_M$ instead of $\T$.  We also may abuse notation and use either $\T$ or $\T_M$ to denote $M$ when confusion will not arise.  A \emph{filling} of a closed triangulated $n$-dimensional manifold $(M,\T_M)$ is a triangulated $(n+1)$-dimensional manifold $(N,\T_N)$ with $\partial N = M$ and $\T_N |_{\partial N} = \T_M$.

If $\T$ is a simplicial complex, the \emph{$k$-skeleton of $\T$}, denoted $\T^{(k)}$, will refer to the subcomplex of $\T$ consisting of all simplices of dimension at most $k$.
A \emph{facet} of a triangulation is a simplex of maximal dimension.
For any triangulation $\T$, the notation $|\T|$ will refer to the number of facets in the triangulation. 
In the case of a triangulated manifold, this will be used as a discrete analogue of volume.

The \emph{systole} of a Riemannian manifold $(M,g)$, denoted $\sys_g(M)$, is the length of the shortest non-contractible loop in $M$. 
The \emph{homological systole} of a Riemannian manifold $(M,g)$, denoted $\sys^{H}_g(M)$, is the length of the shortest homologically nontrivial loop in $M$. 

If $p$ is an edge path in the triangulated manifold $(M,\T)$, the \emph{discrete length} of $p$, denoted $l_{\T}(p)$, will be the number of edges in $p$. 
The \emph{discrete systole} of a triangulated manifold $\T$, denoted $\sys_{\T}(M)$, will refer to the discrete length of the shortest non-contractible edge loop in $\T$. 
The \emph{discrete homological systole}, denoted $\sys ^H_{\T}(M)$, is defined analogously.

Let $S \subset \R^m$.
By a \emph{secant vector} in $S$, we mean any $v = t(x-y)$ where $x, y \in S$ and $t \in \R$.  
We define a \emph{secant simplex} of $S$ to be the convex hull (in $\R^m$) of an affinely independent collection of points in $S$.

If $\sigma$ is an $n$-simplex in $\R^m$, we define it's \emph{fullness} to be
\begin{equation*}
  \Theta(\sigma) = \frac{\text{Vol}_n(\sigma)}{(\diam \sigma)^n},
\end{equation*}
where $\text{Vol}_n$ denotes the $n$-dimensional Hausdorff volume in $\R^m$.
We also note that the diameter, $\diam \sigma$, is the length of the longest side in this case.

Let $P$ be an affine subspace in $\R^m$.  
The function $\pi_P: \R^m \rightarrow P$ will always denote orthogonal projection.  
A point $p \in P$ defines the vector space
$$ V_p(P) = \{ v \in \R^m \,| \, p + v \in P \}. $$
We will frequently identify $P$ with $V(P)$ when there is no ambiguity.  
In particular, if $v = p + w$ is a vector lying in $P$ (and with $w$ lying in $V_{P}$), then define $|v|_{P} = |w|_{\R^m}$.

Suppose $P_1, P_2$ are affine subspaces of $\R^m$ and $f: P_1 \to P_2$ is an affine map. 
Then for any $p \in P_1$, $f$ induces a linear transformation
\begin{equation*}
  V(f) \colon V_p(P_1) \to V_{\pi_{P_2}(p)}(P_2).
\end{equation*}
This map does not depend on the choice of $p$ and is uniquely determined by $f$.
We may then consider $V_p(P_1)$ and $V_{\pi_{P_2}(p)}(P_2)$ as linear subspaces of $\R^m$ and likewise consider $V(f)$ as a linear map between subspaces of $\R^m$. 
Thinking of $V(f)$ in this way allows us to speak of $|v - f(v)| := |w - V(f)(w)|$ for $v =w + p\in P_1$.
Throughout Sections \ref{section:discrete-to-riemannian} and \ref{section:whitneyarg}, expressions such as $|v - f(v)|$ for $v \in P_1$ will always be interpreted this way.

If $M^n \subset \R^m$ is a smooth, embedded submanifold of $\R^m$ and if $p \in M$, we may identify the tangent space $T_p M$ with an affine $n$-plane $P_p \subset \R^m$, whose points are of the form $p + v$ where $v \in T_pM$.
Note that there are two ways to view the projection map onto $P_p$: if we view $P_p$ as an affine subspace of $\R^m$ then $\pi_P$ is the orthogonal projection from $\R^m$ onto $P_p$, while if we are considering $P_p = T_p M$ for $p \in M$, $\pi_p$ then denotes the orthogonal projection from $M$ to $P_p$.  In either case, our meaning will always be clear from context.

Note that $\pi_p|_M \colon M \to P_p$ is regular at $p$, and so there exists a neighborhood $U \subset M$ of $p$ so that $\pi_p \colon U \to \pi_p(U)$ is a diffeomorphism.
For $\xi > 0$, let
\begin{equation*}
  P_{p, \xi} = U_\xi(p) \cap P_p.
\end{equation*}
where $U_\xi(p)$ denotes the open $\xi$ ball about $p$ in $\R^m$.  
For $\xi$ sufficiently small, $\pi_p^{-1}(P_{p, \xi}) \cap M \subset U$ and we let
\begin{equation*}
  M_{p, \xi} = \pi_p^{-1}(P_{p,\xi}) \cap M.
\end{equation*}

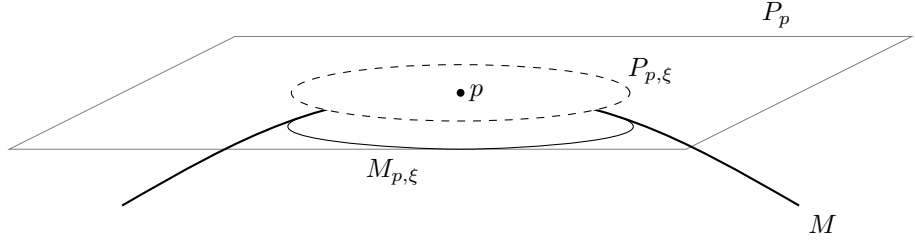
\begin{figure}
  \centering
  \input{Ppxi.tex}
  \caption{$P_{p,\xi}$ and $M_{p,\xi}$}
\end{figure}

If $\sigma^n$ is an $n$-simplex in $\R^m$, it spans an affine $n$-plane, \emph{the plane of $\sigma$}, which we will denote by $P(\sigma)$.

For affine subspaces $P, P'$ of $\R^m$, define the \emph{independence of $P$ and $P'$} to be
\begin{equation*}
  \ind(P,P') = \inf \{|v - \pi_P(v)| \,| \, v \in P', |v| = 1\}.
\end{equation*}
Note that this quantity is symmetric, and does does depend on any particular choices made; in fact, it only depends on $V(P)$ and $V(P')$.
We also note that $\ind(P, P') = 0$ if the planes have a vector in common, and $\ind(P, P') = 1$ if and only if the planes are orthogonal.
\subsection*{Tubular neighborhoods and horizontal tangent vectors}\label{horztangsect}
Suppose $M$ is an embedded submanifold of $\R^m$, and let $U$ be a tubular neighborhood of $M$ in $\R^m$.  
Then the projection map $\pi^* \colon U \to M$ is a Riemannian submersion.
So $TU \cong TU_h \oplus TU_v$ where $TU_h$ is canonically isomorphic to the tangent bundle $TM$ of $M$, and similarly for $TU_v$ and the normal bundle $TN$ of $M$ in $\R^m$.  
We will refer to $TU_h$ as the \emph{horizontal} component of $TU$ and $TU_v$ as the \emph{vertical} component of $TU$.  
So for $q \in U$ and $w \in T_qU$, we will write $w = w_h + w_v$ where $w_h \in T_qU_h$ and $w_v \in T_qU_v$.  
Also note that for any point $q \in U$, the space $T_q U_v$ is equal to the kernel of the derivative of the projection map $D\pi_{\pi^* (q)} = \pi_{\pi^* (q)}$.
So if $w \in T_q U$ and $w = w_h + w_v$, then 
\begin{equation}\label{horizontal equality}
|w_h| = |\pi_{\pi^*(q)}(w)|.
\end{equation}

Now, for any $p \in M$, the map
\begin{equation*}
  D\pi^* \biggr|_{T_pU_h} \colon T_p U_h \to T_p M
\end{equation*}
is the identity. Thus for any $\epsilon > 0$, there is a smaller tubular neighborhood $U' \subset U$ so that, for any $q \in U'$, the map
\begin{equation*}
  D\pi^* \biggr|_{T_qU_h'} \colon T_q U'_h \to T_{\pi^*(q)} M
\end{equation*}
has the property that, for any $w \in T_q U'_h$,
\begin{equation}\label{nearid}
  \frac{1}{\sqrt{3/2}} |w| \leq |D\pi^*(w)| \leq \sqrt{3/2}|w|.
\end{equation}


\section{Encoding a smooth triangulation}\label{section:discrete-to-riemannian}

In this section we prove Theorem \ref{thm:discrete-to-riemannian}. Namely, given a smooth Riemannian manifold $(M,g)$, 
we want to construct a triangulation where the ratio 
$$\frac{\sup_{e\subset \T}\{l_g(e)\}}{\inf_{\sigma \subset \T}\{\sqrt[n]{\vol_g(\sigma)}\}}$$
is uniformly bounded above by a constant $ \delta _n$ which only depends on the dimension of $M$.

We do this by using the following result, which will be proved in Section \ref{section:whitneyarg}.
 \begin{thm}\label{euctriangulation}
  Let $M$ be a compact $n$-dimensional smooth Riemannian submanifold of $\R^m$.
  Then there is an $n$-dimensional simplicial complex $T \subset \R^m$ with the
  following properties:
  \begin{enumerate}
    \item Each simplex of $T$ is a secant simplex of $M$
    \item $T$ is contained in a tubular neighborhood of $M$.
      The projection $\pi^*$ from this neighborhood onto $M$ induces a homeomorphism $\pi^*: T \to M$.
    \item If $\sigma$ is a simplex of $T$ (of any dimension), then it's fullness is bounded below by $\Theta_{n,m}$, which depends only on the dimensions of the manifold and the ambient space.
    \item For any $n$-simplex $\sigma$ of $T$, point $q \in \sigma$ and tangent vector $v \in T_q\sigma$, we get that
      \begin{equation}\label{vectproj}
        |\pi_{\pi^*(q)}(v)| \geq \frac{1}{2}|v|,
      \end{equation}
      where $\pi_{\pi^*(q)}$ is the orthogonal projection onto the tangent plane $P_{\pi^*(q)}$.
    \item If $L$ is the length of an edge in $T$, then
      \begin{equation}\label{edgelengthbounds}
         C_{n,m}\bar L \leq L \leq \bar L
      \end{equation}
      for some positive $\bar L$ and constant $C_{n,m}$ depending only on $n$ and $m$.
  \end{enumerate}
\end{thm}

Let $M$ be an $n$-dimensional Riemannian manifold.
By the Nash Isometric Embedding theorem \cite{Nash}, $M$ embeds smoothly and isometrically into $\R^m$, where $m$ depends only on $n$.
Thus we may consider the case where $M$ is a smooth Riemannian submanifold of $\R^m$.
Then applying Theorem \ref{euctriangulation}, we get an $n$-dimensional simplicial complex $T$ contained in a tubular neighborhood of $M$ so that the tubular neighborhood projection $\pi^*$ induces a homeomorphism from $T$ to $M$.

We will proceed in two parts. 
The first part will consist of showing that the restriction of $\pi^*$ to any $n$-simplex of $T$ is bi-Lipschitz with constants that do not depend on the given simplex. 
In the second part we will prove Theorem \ref{thm:riemannian-to-discrete} by using this fact to relate the geometry of $T$ with the geometry of $\pi^*(T)$.

\subsection{$\pi^*$ is bi-Lipschitz on every $n$-simplex of $T$}

Let $\sigma$ be an $n$-simplex of $T$ and suppose $x_1, x_2 \in \sigma$.
Let $p_1 = \pi^*(x_1)$ and $p_2 = \pi^*(x_2)$.

Suppose $p$ is a unit-speed geodesic in $\sigma$ from $x_1$ to $x_2$.
For every $t \in [0, \ell(p)$], $p'(t)$ is a tangent vector in $\sigma$, so by Theorem \ref{euctriangulation} and equation (\ref{horizontal equality}), 
\begin{equation}\label{horzcomp}
  1 = |p'(t)| \geq |p'(t)_h| = |\pi_{\pi^* p(t)}(p'(t))| \geq \frac{1}{2}|p'(t)| = \frac{1}{2}.
\end{equation}
Now $\pi^* \circ p$ is a path on $M$ from $p_1$ to $p_2$, and since $M$ is isometrically embedded in $\R^m$,
\begin{equation}\label{distlength}
  d_M(p_1, p_2 ) \leq \ell(\pi^* \circ p).
\end{equation}
Combining (\ref{nearid}) and (\ref{horzcomp}), we get that
\begin{align*}
  \ell(\pi^* \circ p) &= \int_0^{\ell(p)} |D\pi^*_{p(t)}(p'(t))| \,dt\\
  &= \int_0^{\ell(p)} |D\pi^*_{p(t)}(p'(t)_h)| \,dt\\
  &\leq \sqrt{3/2}\int_0^{\ell(p)} |p'(t)_h| \,dt\\
  &\leq \sqrt{3/2}\int_0^{\ell(p)} \,dt\\
  &= \sqrt{3/2}\cdot \ell(p).
\end{align*}
Combining the above with (\ref{distlength}) gives that
\begin{equation}\label{upperbd}
  d_M(p_1, p_2) \leq \sqrt{3/2} \cdot |x_1 - x_2|.
\end{equation}

Now suppose $\gamma$ is a unit-speed geodesic in $M$ from $p_1$ to $p_2$ so that $d_M(p_1, p_2) = \ell(\gamma)$.
Then $(\pi^*)^{-1} \circ \gamma$ is a piecewise smooth path in $T$ from $x_1$ to $x_2$.
We may take a partition of the interval $[0, \ell(\gamma)]$ into
\begin{equation*}
  0 = a_0 <  a_1 <  a_2 < \cdots  a_N = \ell(\gamma)
\end{equation*}
so that for each $i$, $(\pi^*)^{-1} \circ \gamma([a_i, a_{i+1}]) \subset \sigma_i$ where $\sigma_i$ is an $n$-simplex of $T$.
Let $\gamma|_{[a_i, a_{i+1}]} = \gamma_i$ and let $(\pi^*)^{-1}\circ \gamma(a_i) = b_i$.

Then $(\pi^*)^{-1} \circ \gamma_i$ is a path in $\sigma_i$ and for every $t \in [a_i, a_{i+1}]$, $(D\pi^*)^{-1}(\gamma_i'(t))$ is a tangent vector in $\sigma_i$. So for every $t \in [a_i, a_{i+1}]$, \eqref{horzcomp} gives that
\begin{equation}\label{pihorzcomp}
  |(D\pi^*)^{-1}(\gamma_i'(t))| \leq 2|[(D\pi^*)^{-1}(\gamma_i'(t))]_h|.
\end{equation}
Suppose $D\pi^*(v) = w$ for $v \in T_p \sigma_i$.
Then $w = D\pi^*(v_h)$ and \ref{nearid} gives that
\begin{equation*}
  |v_h| \leq \sqrt{3/2} |w|.
\end{equation*}
Using the above, we get that, for any $t \in [a_i, a_{i+1}]$,
\begin{equation}\label{piinvh}
  |[(D\pi^*)^{-1}(\gamma_i'(t))]_h| \leq \sqrt{3/2} |\gamma_i'(t)|.
\end{equation}
Combining \eqref{pihorzcomp} and \eqref{piinvh} gives that
\begin{align*}
  |b_i - b_{i+1}| &\leq \ell((\pi^*)^{-1} \circ \gamma)\\
  &= \int_0^{a_{i+1}-a_i} |(D\pi^*)^{-1}(\gamma_i'(t))| \,dt\\
  &\leq 2\int_0^{a_{i+1}-a_i} |[(D\pi^*)^{-1}(\gamma_i'(t))]_h| \,dt\\
  &\leq 2\sqrt{3/2}\int_0^{a_{i+1}-a_i} |\gamma_i'(t)| \,dt\\
  &\leq 2\sqrt{3/2}\cdot d_M(\gamma(a_i), \gamma(a_{i+1})).
\end{align*}

Since $\gamma$ is a minimizing geodesic, $\sum d_M(\gamma(a_i),\gamma(a_{i+1})) = d_M(p_1,p_2)$.
So
\begin{equation}\label{lowerbd}
  |x_1 - x_2| \leq \sum_0^{N+1} |b_i - b_{i+1}| \leq 2\sqrt{3/2} \sum_0^{N+1}d_M(\gamma(a_i), \gamma(a_{i+1})) = 2\sqrt{3/2} \cdot d_M(p_1, p_2).
\end{equation}
Combining \eqref{upperbd} and \eqref{lowerbd} gives that, for any $x_1, x_2 \in \sigma$,
\begin{equation}\label{bilip}
  \frac{1}{2\sqrt{3/2}}\cdot |x_1 - x_2| \leq d_M(\pi^*(x_1), \pi^*(x_2)) \leq \sqrt{3/2}\cdot |x_1 - x_2|.
\end{equation}

\subsection{The geometry of $\pi^*(T)$}

By the previous section, if $e$ is an edge of the complex $T$, then
\begin{equation}\label{edge-bilip}
  \frac{1}{2\sqrt{3/2}}\ell(e) \leq \ell(\pi^*(e)) \leq \sqrt{3/2}\ell(e).
\end{equation}
Letting $\sigma$ be an $n$-simplex of $T$, it then follows that
\begin{equation}\label{volcomp1}
  \left(\frac{1}{2\sqrt{3/2}}\right)^n\vol_n(\sigma) \leq \vol_M(\pi^*(\sigma)) \leq (\sqrt{3/2})^n\vol_n(\sigma).
\end{equation}

\subsection{Proof of Theorem \ref{thm:discrete-to-riemannian}}

Let $\sigma \in \pi^*(T)$ be the simplex in $M$ for which $\vol_M(\sigma)$ is minimal among all simplices in $\pi^*(T)$.
Let $E$ be the length of the longest edge in $\pi^*(T)$, and let $L$ be the length of the longest edge in $\sigma$.  
So there exist edges $e,l \in T$ such that $E = \ell(\pi^*(e))$ and $L = \ell(\pi^*(l))$.  
By equations \eqref{edge-bilip} and \eqref{edgelengthbounds} we have that
\begin{equation*}
E = \ell(\pi^*(e)) \leq \sqrt{3/2} \ell(e) \leq \sqrt{3/2} \bar{L} \leq \sqrt{3/2} \frac{L}{C_{n,m}}
\end{equation*}
So,
\begin{equation}\label{edgebd}
  E^n \leq (\sqrt{3/2})^n \frac{L^n}{C^n_{n,m}} .
\end{equation}
By \eqref{bilip}, we have that
\begin{equation}\label{lengthbd}
 L \leq \sqrt{3/2}\diam \left((\pi^*)^{-1} \sigma\right).
\end{equation}
Using \eqref{edgebd}, \eqref{lengthbd}, \eqref{volcomp1}, and Theorem \ref{euctriangulation} (3) we obtain
\begin{align*}
  \frac{\vol_M(\sigma)}{E^n} &\geq \left( \frac{C_{n,m}}{\sqrt{3/2}} \right)^n \frac{\vol_M(\sigma)}{L^n}\\
  &\geq \frac{C^n_{n,m}}{2^n(\sqrt{3/2})^{2n}} \frac{\vol_n(\pi^*)^{-1}(\sigma)}{L^n}\\
  &\geq \frac{C^n_{n,m}}{2^n(\sqrt{3/2})^{3n}} \frac{\vol_n(\pi^*)^{-1}(\sigma)}{(\diam (\pi^*)^{-1}\sigma)^n}\\
  &\geq \frac{C^n_{n,m}}{(\sqrt{27/2})^n} \Theta_{n,m}.
\end{align*}
Since $m$ depends only on $n$, we have proved Theorem \ref{thm:riemannian-to-discrete}, with the value
\begin{equation*}
  \delta_n = \frac{3 \sqrt{3/2}}{C_{n,m}\Theta^{1/n}_{n,m}}.
\end{equation*}


\section{Whitney's triangulation procedure}\label{section:whitneyarg}

In this Section, we analyze in detail Whitney's triangulation argument, and establish the technical Theorem \ref{euctriangulation}.
This result was mentioned at the beginning of the previous Section \ref{section:discrete-to-riemannian}, where it was used
to prove Theorem \ref{thm:discrete-to-riemannian}. 

This Section very closely parallels the procedure outlined by Whitney \cite{Whitney} to prove smooth manifolds are triangulable,
and proofs which follow directly from \cite{Whitney} are omitted for brevity and readability.
The idea is to cubulate $\R^m$ and take the cubulation's barycentric subdivision.
We then move the vertices to ensure $M$ is far from the $(m-n-1)$-skeleton of the complex.
Using the poset of intersections of simplices of dimensions $(m-n),\dots,m$ gives us a simplicial complex that sits inside a tubular neighborhood of $M$.
We then prove the tubular neighborhood projection induces a diffeomorphism onto $M$.

We are using Whitney's construction with a slightly different purpose in mind: to produce triangulations whose fullness only depends on the dimension of the manifold in question. 
This means some minor modifications need to be made, mostly relating to the choosing of certain quantities.  

Since $T$ denotes the triangulation in Theorem \ref{euctriangulation}, throughout this Section we use $P_p$ and $P_p^*$ (where $p \in M$) to denote the tangent and normal planes $T_pM$ and $N_pM$, respectively.

\subsection{Quantities used in the proof}
First, some remarks on the barycentric subdivision of a cube in $\R^m$ with side length $h$.
This breaks up the cube into $2^mm!$ simplices, where the longest edge is the one connecting a corner of a cube to its center, which is of length $h\sqrt{m}/2$. 
Thus each $m$ simplex has volume $h^m/(2^mm!)$ and diameter $h\sqrt{m}/2$, which gives a fullness of $1/(m!m^{m/2})$.
So any barycentric subdivision of a cubical subdivision of $\R^m$ forms a triangulation of $\R^m$ where the fullness of each simplex is $2\Theta_0 = 1/(m!m^{m/2})$.
Let $N$ be the maximum number of simplices in any star of any vertex in such a triangulation.

Let $\rho^*$ be given by Lemma \ref{aII,16a}.

By Lemma \ref{14c}, we may choose $\rho_0 < \text{min} \{1/(4m^{1/2}), 2\rho^*/m^{1/2} \}$ 
so that for any $m$-simplex $\sigma = \langle p_0, \cdots, p_m \rangle$, if
$\Theta(\sigma) \geq 2\Theta_0$, and $|q_i - p_i| \leq \rho_0 \cdot \diam \sigma$,
then $\tau = \langle q_0, \cdots, q_m \rangle $ is a simplex, with $\Theta(\tau) \geq \Theta_0$.

We choose $\rho_1$ as in Lemma \ref{rho1} so that
\begin{equation}\label{rho1small}
  0 <  \rho_1 < \frac{4}{\rho_0 \sqrt[s]{2}},
\end{equation}
where $s = m - n \geq 1$ will always denote the codimension of the embedding.
We then define the following constants, which depend only on $n$ and $m$:
\begin{equation}\label{17.2}
  \rho = \frac{\rho_0 \rho_1}{4}, \quad \alpha_r = \frac{\rho^r \rho_0 \rho_1}{2} \; (\text{for } 0 \leq r \leq s-2), \quad \alpha = \frac{\alpha_{s-1}}{4},\\
\end{equation}
\begin{equation}\label{17.3}
  \beta = \frac{\Theta_0 \alpha}{m^{1/2}N}, \quad \Theta_1 = \frac{\beta^n}{2^n}, \quad \gamma = \frac{(n-1)!\Theta_1\beta}{2}.
\end{equation}
The choice of $\rho_1$ in \eqref{rho1small} ensures that $\rho < 1/2$, $\alpha_{r+1} < \alpha_r$, and $\alpha < 1/4$.  Also, it is important to emphasize that $\beta$ only depends on $n$ and $m$, as the constants $\Theta_{m,n}'$ and $C_{n,m}$ from Theorem \ref{euctriangulation} will both depend on $\beta$.

Choose $\rho_0' \leq 1/4$ by Lemma \ref{14c} using $n$, $\Theta_1$, 
$\Theta_1/2$
in place of $r$, $\Theta_0$, $\epsilon$. Then let
\begin{equation}\label{17.4}
  \lambda = \inf\, \left\{ \frac{\alpha\gamma}{128},\,\frac{\rho_0'\alpha\beta}{8} \right\}.
\end{equation}
Note that the above constants all depend on $\rho_0$, $\rho_1$, $\rho_0'$, and $\rho^*$, which only depend on $n$ and $m$.  But these four constants can be chosen as small as we like, a fact that will be used a few times throughout the course of this section.

Using the notation of Section \ref{horztangsect}, we may also choose $\delta_0$ in Theorem \ref{10A} so that $U^* \subset U'$.

We then choose $\xi_0$ by Lemma \ref{8a}.

Now choose $\xi_1 \leq \xi_0$ in Lemma \ref{8b}. Finally, we define more constants:
\begin{equation}\label{17.5}
  \xi = \inf\, \left\{ \xi_1,\, \frac{\alpha \delta_0}{3\lambda} \right\},\quad \delta = \frac{\xi}{8}, \quad h = \frac{2\delta}{m^{1/2}},
\end{equation}
\begin{equation}\label{17.6}
  a = 2\alpha\delta, \quad b = \beta\delta, \quad c = \gamma \delta.
\end{equation}

\begin{rem}
We began this section by cubulating $\R^m$ by cubes with edge length $h$.  We then used Lemma's \ref{14c}, \ref{aII,16a}, and \ref{rho1} to find constants $\rho_0$, $\rho^*,$ and $\rho_1$.  It is important to note that,  while these constants are defined using a cubulation of $\mathbb R^n$, they do \emph{not} depend on the size $h$ of the cubulation. Indeed, the attentive reader will notice that these constants are set up to be scale invariant. As a result, it is not circular to redefine $h$ in equation \eqref{17.5}.
\end{rem}

\subsection{The complexes $L$ and $L^*$}
First, we let $L_0$ be a cubical subdivision of $\R^m$ with cubes of side length
$h$, and let $L$ be the barycentric subdivision of $L_0$. Then each edge of
$L$ has length at least $h/2$, and we may choose $h$ small enough so that the $m$-simplices have diameter $ \leq \delta$.

Suppose $L$ has vertices $\{p_i \}_{i \in \mathbb{N}}$. We are going to recursively construct a new 
triangulation of $\R^m$, $L^*$, whose $(s-1)$-skeleton is sufficiently far away 
from $M$. The vertices of this new complex will be denoted $\{p_i^* \}_{i \in \mathbb{M}}$. We 
will choose these vertices so that
\begin{equation}\label{18.1}
  |p_i^* - p_i| < \rho_0 \delta
\end{equation}
for all $i$. 
By Lemma \ref{aII,16a} and the definition of $\rho_0$, we get a new triangulation of $\R^m$. 
Since 
$\rho_0 < 1/(4m^{1/2}$), we get that $\rho_0 \delta < h/8$. The diameter of 
any simplex of $L$ is at least $h/2$, so any simplex $\tau$ of $L^*$ will satisfy
\begin{align*}
  \diam(\tau) &\geq h/2 - 2 \rho_0\delta\\
              &> h/2 - h/4\\
              &= h/4.
\end{align*}
Because the diameter of each simplex of $L$ is at most $\delta$, we also have that
\begin{align*}
  \diam(\tau) &< \delta + 2\rho_0\delta\\
              &< \delta + 2\frac{1}{4m^{1/2}}\delta\\
              &< 2 \delta.
\end{align*}
Combining the above, we obtain
\begin{equation}\label{18.2}
  h/4 < \diam(\tau) < 2\delta.
\end{equation}

By our choice of $\rho_0$ and \eqref{14.6} we have that, for all simplices $\tau$ of $L^*$ of dimension at least 1,
\begin{equation}\label{18.2b}
  \Theta(\tau) \geq \Theta_0.
\end{equation}

The following Lemma is proved in \cite{Whitney} pages 129-130.  
We include most of the proof below because some of the equations and techniques are necessary to prove Theorem \ref{euctriangulation}.

\begin{lem}\label{lemma:missingLowDimensionalSimplices}
Suppose $p_1^*, \dots, p_{i-1}^*$ have been found so that the complex 
$L_{i-1}^*$ with these vertices satisfies
\begin{equation}\label{18.3}
  \dist(M, \tau^r) > \alpha_r \delta 
\end{equation}
for all $\tau^r \in L^*_{i-1}$ and $r \leq s-1$.
Then there exists $p_i^*$ so that $|p_i^* - p_i| < \rho_0 \delta$ and \eqref{18.3} holds for $L_i^*$.
\end{lem}

\begin{proof}
\noindent {\bf Case 1.} If $\dist(M, p_i) \geq 3\delta$, then we set $p_i^* = p_i$.
By \eqref{18.2} we get that $\dist(M, \tau) > \delta \geq \alpha_r \delta$ and \eqref{18.3} holds for $L_i^*$.

\noindent {\bf Case 2.} If there is a point $p \in M$ with $|p - p_i| < 3\delta$, let $P_0 =
P_p$ be the tangent plane of $M$ at $p$. Let $\tau'_1, \dots, \tau'_\nu$ be the
simplices of $L_{i-1}^*$ of dimension at most $s-2$ so that $\tau_j =
p_i^*\tau_j'$ will be a simplex of $L_i^*$. Since the star of any vertex of $L$
or $L^*_{i}$ cannot have more than $N$ simplices, we have that $\nu < N$. For
$j \geq 1$, let $P_j$ be the affine plane spanned by $\tau_j'$ and $P_0$. Its
dimension is at most $(s - 2) + n + 1 = m -1 < m$. Let
\begin{equation}\label{18.5}
  Q_j = U_{\rho_0 \delta}(p_i) \cap U_{\rho_1 \rho_0 \delta}(P_j), \text{ for } j = 0, \dots, \nu.
\end{equation}
Note that each $Q_j$ is (possibly strictly) contained in the part of the ball $U_{\rho_0\delta}(p_i)$ between a pair of parallel $(m-1)$-planes each at distance $\rho_1 \rho_0 \delta$ from $P_j$, and at distance $2 \rho_1 \rho_0 \delta$ from each other.
Thus, by the definition of $\rho_1$, we have that
\begin{equation*}
  \vol(Q_j) < \frac{\vol(U_{\rho_0 \delta}(p_i))}{N}.
\end{equation*}
The total volume of the sets $Q_0, \dots, Q_\nu$ is less than the volume of 
the ball $U_{\rho_0 \delta}(p_i)$; thus we can find a point $p_i^*$ so that 
both \eqref{18.1} holds and 
\begin{equation}\label{18.6}
  \dist(p_i^*, P_j) > \rho_1 \rho_0 \delta \quad (j = 0,..., \nu).
\end{equation}

We now need the following Claim, whose proof can be found in \cite{Whitney} pg. 130

{\bf Claim:} 
\begin{equation}\label{18.7}
  \dist(\tau'_j, P_0) > 2 \alpha_{r-1} \delta/3 \text{ if $\dim(\tau_j') = r - 1$ and $r < s$}.
\end{equation}



\emph{Proof Continued:}  Using Lemma \ref{aII,14b} and equations \eqref{18.6} and \eqref{17.2}, we have that
\begin{align}
  \dist(\tau_j, P_p) & \geq \frac{\dist(\tau_j', P_p) \cdot \dist(p_i^*, P_ j)}{\diam(\tau_j)}\\
                     &> \left( \frac{2\alpha_{r-1} \delta}{3} \right) \left( \frac{\rho_1 \rho_0 \delta}{2\delta} \right)\\
                     &= \frac{4\alpha_{r-1}\rho\delta}{3} = \frac{4\alpha_r\delta}{3}.
\end{align}

Using the above and \eqref{8.6} gives
\begin{equation*}
  \dist(M_{p,\xi}, \tau_j) \geq \dist(\t_j, P_0) - \lambda \xi > \frac{4\alpha_r\delta}{3} - \frac{\alpha_r\delta}{3} = \alpha_r\delta.
\end{equation*}
Applying \eqref{8.3}, \eqref{18.1}, and \eqref{18.2}, we get
\begin{align*}
  \dist(M \setminus M_{p,\xi}, \tau^r) &\geq \dist(M \setminus M_{p,\xi}, p) - |p - p_i| - |p_i - p_i^*| - \diam(\tau^r)\\
                                       &>\xi - 3\delta - \rho_0 \delta - 2\delta\\
                                       &>\delta.
\end{align*}
Thus, since $\dist(\t^r, M) = \text{min} \left\{ \dist(\t^r, M_{p, \xi}), \dist(\t^r, M \setminus M_{p, \xi}) \right\}$, we have verified \eqref{18.3} for $\tau^r = \tau_j$, $j \geq 1$.
Using $j = 0$ in \eqref{18.6} and \eqref{8.6} then gives
\begin{align*}
  \dist(p_i^*, M_{p,\xi}) &\geq \dist(p_i^*, P_p) - \lambda \xi\\
                          &> \rho_1\rho_0\delta - \frac{\alpha_0\delta}{3}\\
                          &> \alpha_0\delta.
\end{align*}
Again by \eqref{8.3} and \eqref{18.1}, we get
\begin{align*}
  \dist(M \setminus M_{p,\xi}, p_i^*) &\geq \dist(M \setminus M_{p,\xi}, p) - |p - p_i| - |p_i - p_i^*|\\
                                       &>\xi - 3\delta - \rho_0 \delta\\
                                       &>2\delta > \alpha_0\delta.
\end{align*}
So \eqref{18.3} holds for $\tau^r = p_i^*$ and thus \eqref{18.3} holds in all cases.
In particular, if $L^{*(s-1)}$ denotes the $(s-1)$-skeleton of $L^*$, then
\begin{equation}\label{18.4}
  \dist(M, L^{*(s-1)}) > \alpha_{s-1} \delta = 4\alpha\delta = 2a.
\end{equation}
\end{proof}

\subsection{The intersections of $M$ with $L^*$}\label{intersections}
The following six claims all deal with how $M$ and its tangent planes intersect the simplices of $L^*$.  
In the rest of our arguments, we only use claims \ref{4.5.5} and \ref{4.5.6}. Nevertheless, we include claims \ref{4.5.1} through \ref{4.5.4} in order to motivate the last two claims.
These first four claims essentially state that $M$ intersects a face $\tau$ of $L^*$ near a point $p$ if and only if the tangent plane at $p$ intersects $\tau$ -- and in this case, the tangent plane at $p$ intersects $\tau$ transversely.  
Claims \ref{4.5.5} and \ref{4.5.6} deal specifically with how $M$ intersects the $s=m-n$ skeleton of $T$.
The proofs of all six claims can be found in \cite{Whitney} pages 130-131.

\begin{claim}\label{4.5.1}
  For any $p \in M$ and $r$-simplex $\tau^r$ of $L^*$, we have that
  \begin{equation}\label{19.1}
    \dist(P_p, \tau^r) > a \quad \text{ if $\tau^r \subset U_{7\delta}(p)$ \text{ and } $r \leq s-1$}.
  \end{equation}
\end{claim}

\begin{claim}\label{4.5.2}
  If $M$ intersects $\tau^r$, $p \in M$, and $\tau^r \subset
  U_{7\delta}(p)$, then $P_p$ intersects $\tau^r$. 
\end{claim}

\begin{claim}\label{4.5.3}
  If $r = s$ in Claim \ref{4.5.2} and if $P(\tau^s)$ is the affine plane
  spanned by $\tau^s$, then 
  \begin{equation}\label{19.2}
    \ind(P_p, P(\tau^s)) > \frac{a}{2\delta} = \alpha. 
  \end{equation}
\end{claim}

\begin{claim}\label{4.5.4}
  If $p \in M$, $\tau^r \subset U_{7\delta}(p)$, and $P_p$ intersects $\tau^r$, then $r \geq s$ and $M_{p,\xi}$ intersects $\tau^r$.
\end{claim}

\begin{claim}\label{4.5.5}
  $M$ intersects any $\tau^s$ in at most one point.
\end{claim}

\begin{claim}\label{4.5.6}
  If $M$ intersects $\tau^r = \langle q_0,\cdots, q_r \rangle$, then for each $k$, $M$
  intersects some $s$-face of $\tau^r$ containing $q_k$.
\end{claim}
\subsection{The complex $K$}
For every simplex $\tau$ of $L^*$ that intersects $M$, we will choose a point
$\psi(\tau)$ inside $\tau$, and create a simplicial complex which is the order
complex of the poset of faces of simplices of $L^*$ that intersect $M$. (Note
that this produces a complex of dimension $n$).

If a simplex $\tau^s$ intersects $M$, it does so at a single point; let
$\psi(\tau^s)$ be that point. If $\tau^r$ intersects $M$, for $r > s$, let
$\tau^s_1,\dots,\tau^s_k$ be the $s$-faces of $\tau^r$ that intersect $M$
(which exist by the previous subsection).
Then let
\begin{equation}\label{20.2}
  \psi(\tau^r) = \frac{1}{k}\sum_0^k \psi(\tau_i^s).
\end{equation}

For $\tau^s = \langle q_0,\dots, q_s \rangle$ intersecting $M$, we have that
\begin{equation}\label{20.3}
  \mu_k > 2\alpha 
\end{equation}
where $0 \leq k \leq s$ and $\psi(\tau^s) = \sum \mu_i q_i$.
To see this, let $\tau_k$ be the $(s-1)$-face opposite $q_k$. 
Let $A_k$ and $A_k'$ be the altitudes from $q_k$ and $\psi(\tau^s)$ respectively to $P(\tau_k)$.
Since $\psi(\tau^s) \in M$, \eqref{18.2} and \eqref{18.3} give that
\begin{equation*}
  \mu_k = \frac{A_k'}{A_k} > \frac{\alpha_{s-1} \delta}{2 \delta} = \frac{4 \alpha \delta}{2 \delta} = 2\alpha.
\end{equation*}

\begin{figure}
\centering
\input{complexK.tex}
\caption{The complex $K$}\label{complexK}
\end{figure}
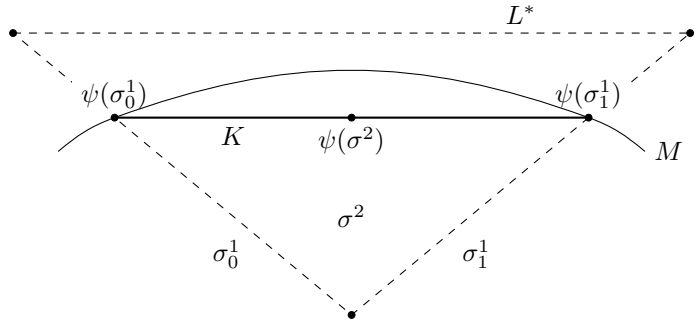

Also, if $\tau^r = \langle q_0,\cdots, q_r\rangle$ intersects $M$, then
\begin{equation}\label{20.4}
  \mu_k > \frac{2\alpha}{N} 
\end{equation}
where, again, $0 \leq k \leq r$ and $\psi(\tau^r) = \sum \mu_i q_i$.

Given $k$, let $\tau^s$ be an $s$-face of $\tau^r$ containing $q_k$, which intersects $M$ by Claim \ref{4.5.6}  of the previous subsection.
By \eqref{20.3}, the barycentric coordinate $\mu'$ of $\psi(\tau^s)$ corresponding to $q_k$ is at least $2\alpha$.
By \eqref{20.2}, $\mu_k$ is the average of at most $N$ barycentric coordinates, one of which is $\mu'$.
Thus, \eqref{20.4} holds.

Since $K$ is an order complex, each simplex $\sigma$ of $K$ has a natural order.
Let $\alt(\sigma)$ be the altitude from the last (highest dimension) vertex of $\sigma$.
We wish to show that
\begin{equation}\label{20.5}
  \alt(\sigma^r) \geq rb.
\end{equation}
Let $\sigma^r = \langle \psi(\tau_0),\cdots,\psi(\tau^r) \rangle$ be a simplex of $K$, with vertices in ascending order. 
Let $\sigma^{r-1}$ be the $(r-1)$-face of $\sigma^r$ opposite $\psi(\tau_r)$.
Then $\sigma^{r-1}$ lies in $\tau_{r-1}$.

If $\dim(\tau_r) = t \geq r$, Lemma \ref{14b}, \eqref{20.4}, \eqref{18.2}, and \eqref{17.5} give
\begin{align*}
  \alt(\sigma^r) &\geq \dist(\psi(\tau_r), \partial \tau_r)\\
                 &\geq t!\Theta(\tau_r)\diam(\tau_r)\inf\{\mu_0,\dots,\mu_r\}\\
                 &\geq r!\Theta_0 \left( \frac{h}{4} \right) \left( \frac{2\alpha}{N}\right) \\
                 &\geq rb.
\end{align*}
Because $\vol_r(\sigma^r) = \alt(\sigma^r) \vol_{r-1}(\sigma^{r-1})/r$, induction gives $\vol_r(\sigma^r) \geq b^r$.
Thus, by $\eqref{18.2}$ and the above,
\begin{equation}\label{20.6}
  \Theta(\sigma) \geq \frac{b^r}{2\delta^r} = \frac{\beta^r}{2^r} \geq \frac{\beta^n}{2^n} = \Theta_1.
\end{equation}

\subsection{Embedding simplices in $M$}
The following claim is proved in \cite{Whitney} on pg. 132.  
\begin{claim}
  If $\sigma$ is a simplex of $K$ and $\sigma \subset U_{6\delta}(p)$ for $p \in M$ then 
  \begin{equation}\label{21.1}
    \sigma \subset U_{\lambda\xi}(P_{p,\xi}).
  \end{equation}
\end{claim}


Combining the above claim and \eqref{10.4} then gives that 
\begin{equation}\label{21.2}
  K \subset U_{2\lambda\xi}(M) \text{ and }|\pi^*(q) - q| < 4\lambda\xi \qquad \forall \, q \in K.
\end{equation}

The following Lemma is proved by Whitney on page 132 of \cite{Whitney}.  
We include the proof here because equations \eqref{fullsigma'} and \eqref{21.5}
are needed to prove Theorem \ref{euctriangulation}.

\begin{lem}[IV, 21a \cite{Whitney}]\label{21a}
  Let $\sigma = \langle p_0,\cdots , p_n \rangle$ be an $n$-simplex of $K$ (with vertices in increasing order) and let $p_0',\dots,p_n'$ be any points such that
  \begin{equation*}
    |p_i'-p_i| \leq \frac{\lambda \xi}{\alpha} \quad  i = 0,\dots,n.
  \end{equation*}
  Then $\sigma' = \langle p_0', \cdots,  p_n' \rangle$ is a simplex in $U^*$, and $\pi^*$ embeds $\sigma'$ in $M$.
\end{lem}

\begin{proof}
  First, note that $\lambda \xi/\alpha \leq \rho_0' \beta 8 \delta/8 = \rho_0' b$, $\Theta(\sigma) \geq \Theta_1$, and $\diam(\sigma) \geq b$. 
  By choice of $\rho_0'$, we have that 
  \begin{equation}\label{fullsigma'}
    \Theta(\sigma')\geq \frac{\Theta_1}{2}.
  \end{equation}
  By \eqref{21.2} and the definition of $\xi$, we get that $\sigma' \subset U_\eta(M)$, where
  \begin{equation}\label{21.4}
    \eta = \frac{ \lambda\xi}{\alpha} + 2\lambda \xi \leq \frac{3 \lambda \xi}{\alpha} \leq \delta_0.
  \end{equation}
  So $\sigma'$ is a simplex in $U^*$.
  Let $q \in \sigma'$ and suppose $q \in P_p^*$ where $p \in M$; thus $\pi^*(q)=p$. 
  By \eqref{10.4}, \eqref{21.4} and the definition of $\lambda$,
  \begin{equation*}
    |p - q| \leq 2 \left( \frac{3\lambda \xi}{\alpha} \right) = \frac{48 \lambda \delta}{\alpha} \leq 6 \rho_0' \beta \delta < \delta.
  \end{equation*}
  Also, $\lambda \xi/\alpha < \delta$, so $\sigma \subset U_{4\delta}(p)$ since $\diam(\sigma) < 2\delta$.
  Equation \eqref{21.1} implies that $\sigma \subset U_{\lambda\xi}(P_p)$ and thus $\sigma \subset U_{2\lambda\xi/\alpha}(P_p)$.
  \eqref{20.5} gives that $|p_i - p_0| \geq b$, thus
  \begin{equation}\label{21.5}
    |p_i' - p_0'| \geq b - \frac{2\lambda\xi}{\alpha} \geq b - 2\rho_0'b \geq b/2.
  \end{equation}

  By Lemma \ref{15c}, if $u$ is a unit vector in $P(\sigma')$, we get that
  \begin{equation}\label{uproj}
    |u - \pi_p (u)| \leq \frac{2(2\lambda\xi/\alpha)}{(n-1)!(\Theta_1/2)(b/2)} = \frac{64 \lambda}{\alpha\gamma} \leq \frac{1}{2}.
  \end{equation}
  Since $P_p^*$ is normal to $P_p$, $u$ cannot be in $P_p^*$. 
  Also, $\pi^*$ maps any nonzero vector in $\sigma'$ at $q$ to a nonzero vector and thus $\pi^*$ is regular at $q$.
  If $q'$ is any other point of $\sigma'$, letting 
  \begin{equation*}
   u = \frac{q' - q}{|q'-q|}
  \end{equation*}
  in the above gives that $q' \notin P_p^*$, and thus $\pi^*(q) \neq \pi^*(q')$.
\end{proof}
\subsection{The complexes $K_p$}

For $p \in M$, let $L_p^*$ be the subcomplex of $L^*$ containing simplices which touch $\bar{U}_{4\delta}(p)$, together with their faces.
Then 
\begin{equation}\label{22.1}
  L_p^* \subset U_{6\delta}(p).
\end{equation}
Let $K_p''$ be the complex in $P_p$ formed by the intersections of $P_p$ with the simplices of $L_p^*$ and let $K_p'$ be the barycentric subdivision of $K_p''$.
By the earlier discussion, $P_p$ intersects a simplex of $L_p^*$ if and only if $M$ does.
Let $K_p$ be the subcomplex of $K$ containing all simplices with vertices $\psi(\tau)$ for each simplex $\tau$ in $L_p^*$.
Then there is a one-to-one correspondence $\phi_p$ of the vertices in $K_p$ onto the vertices of $K_p'$ which defines a simplicial mapping $\phi_p$ that is an isomorphism of $K_p$ onto $K_p'$.

The following three claims are all proved in \cite{Whitney} pg. 133.

\begin{claim}
  If $q \in K_p$, then
  \begin{equation}\label{22.2}
    |\phi_p(q) - q| < \frac{\lambda \xi}{\alpha}.
  \end{equation}
\end{claim}


\begin{claim}
  \begin{equation}\label{22.3}
    K \cap U_{2\delta}(p) \subset K_p.
  \end{equation}
\end{claim}


Choose an orientation for $P_p$ and use this to orient all $n$-simplices of $K_p'$.
This makes $K_p'$ an oriented $n$-psuedomanifold with boundary; \eqref{22.1} and the definition of $L_p^*$ then give that
\begin{equation}\label{22.4}
  K_p' \subset U_{6\delta}(p),\quad \partial K_p' \subset P_p - \bar{U}_{4\delta}(p).
\end{equation}

Define a mapping $\pi_p^*$ of $K_p$ into $P_p$ as follows: Each $q \in K_p$ is in a unique $P_{p'}^*$, so $p' = \pi^*(q)$.
By \eqref{22.1}, $|q-p| \leq 6\delta$ and by \eqref{21.2}, $|p'-q|<4\lambda\xi<\delta$; thus $|p-p'|<\xi$.

By Lemma \ref{10a}, $P_{p'}^*$ intersects $P_{p'}$ in a unique point, which we call $\pi_p^*(q)$.

\begin{claim}
  If $q \in K_p$, then
  \begin{equation}\label{22.5}
    |\pi_p^*(q)-q| < 6\lambda\xi.
  \end{equation}
\end{claim}

\subsection{Proof of Theorem \ref{euctriangulation}}
Let $p \in M$ and choose an orientation of $P_p$, which gives an orientation of both $K_p'$ and $K_p$.
Define the set $R_p$ to be all $q \in K_p$ such that $\pi_p^*(q) \in P_{p, 3\delta}$.
To prove Theorem \ref{euctriangulation} we need the following Lemma, whose proof can be found in \cite{Whitney} page 134.

\begin{lem}[IV, 23b \cite{Whitney}]\label{23b}
  For each $p \in M$, the map
  \begin{equation*}
    \pi_p^* \colon R_p \to P_{p,3\delta}
  \end{equation*}
  is bijective.
\end{lem}

\begin{proof}[Proof of Theorem \ref{euctriangulation}]
  Let $p \in M$, then by Lemma \ref{23b}, there is $q \in K_p$ with $\pi_p^*(q) = p$.
  Thus $\pi^*(q) = p$ and $\pi^*$ is onto.
  Suppose $\pi^*(q') = p$ for $q' \in K$. 
  By \eqref{21.2}, we have that $|q'-p| < 4\lambda \xi < \delta$. 
  Using \eqref{22.3}, $q' \in K_p$ and applying \eqref{22.5} gives that
  \begin{equation*}
    |\pi_p^*(q') -p| \leq |\pi_p^*(q')-q'| + |q' -p| < 6\lambda \xi + \delta < 3\delta,
  \end{equation*}
  so $q' \in R_p$. By Lemma \ref{23b}, $q' = q$. So $\pi^*$ is bijective.
  Since $\pi^*$ is a bijective map between compact Hausdorff spaces, it is a homeomorphism.
  From Theorem \ref{10A}, we know that $\pi^*$ is smooth.
  By the proof of Lemma \ref{21a}, we know $\pi^*$ always has positive Jacobian.
  Therefore, $\pi^*$ is a diffeomorphism.

  Now, using Lemma \ref{21a}, \eqref{21.2} and the fact that $4\lambda\xi < \lambda \xi /\alpha$ gives that the simplicial complex in $\R^m$ whose vertices are $\pi^*(v)$ for each vertex $v$ of $K$ is still homeomorphic to $M$ via $\pi^*$.
  Call this simplicial complex $T$. 
  Then every simplex of $T$ is a secant simplex of $M$.
  By \eqref{fullsigma'}, we get that every simplex of $T$ has fullness at least $\Theta_1/2 := \Theta_{n,m}$, a number which depends only on $n$ and $m$. 

  Let $v \in T_q \sigma$ for $q \in \sigma$ with $\sigma$ a simplex of $T$.
  Then \eqref{uproj} gives that
  \begin{equation*}
    |\pi_{\pi^*(q)} (v)| \geq |v| - |v - \pi_{\pi^*(q)} (v)| \geq |v| - \frac{1}{2}|v| = \frac{1}{2}|v|.
  \end{equation*}

  By \eqref{21.5} and \eqref{21.2},
  \begin{equation*}
    \frac{\beta \delta}{2} = \frac{b}{2} \leq \text{ length of an edge in }T \leq 2\delta + 8\lambda\xi \leq 3\delta.
  \end{equation*}
  If $\bar L = 3\delta$ and $C_{n,m} = \beta/6$, we have that $C_{n,m}$ depends only on $n$ and $m$ and thus \eqref{edgelengthbounds} holds in $T$.
\end{proof}


\section{Encoding a Riemannian metric}\label{section:riemannian-to-discrete}

This section is devoted to the proof of Theorem \ref{thm:riemannian-to-discrete}. We are given a smooth triangulated manifold $(M, \T)$,
and we would like to put a Riemannian metric on $M$ whose geometry captures the combinatorics of the triangulation.

The proof is broken down into four parts.  In the first part we will use a metric $g_s$ which will, in general, have singularities, to produce a Riemannian metric $g_{\d}$ which will depend on a parameter $\d > 0$.  The second step of the proof will be to show that we can choose $\d$ small enough so that $|\text{Vol}_{g_{\d}}(M) - \text{Vol}_{\T}(M)| < \e$.  Letting $g = g_{\d}$ completes the proof of property (1).  We will then give a constructive method to homotope a closed $g_s$-polygonal path $\g$ in $M$ to a closed edge loop $p$ in such a way that 
$$ l_{\T}(p) \leq \kappa_n l_{g_s}(\g). $$
Finally, we will argue that, once $\delta$ is small enough, the above inequality is preserved when considering closed geodesics in the metric $g$ instead of polygonal paths in the metric $g_s$, completing the proof of property (2).

\subsection*{Constructing the Riemannian metric $g_{\d}$ from the singular metric $g_s$}  The singular metric $g_s$ is defined by simply requiring that every facet of $\T$ be isometric to an equilateral $n$-simplex in $\E^n$ with volume 1.  Thus it is clear that $\text{Vol}_{g_s}(M) = \text{Vol}_{\T}(M)$.  The singular set of $g_s$ will be contained in the codimension two skeleton $\T^{(n-2)}$ of $\T$.  A simplex $\s \in \T^{(n-2)}$ will be contained in the singular set of $g_s$ if and only if ``too many or too few" facets of $\T$ meet at $\s$.  For example, if $n=2$, a vertex $v$ is contained in the singular set of $g_s$ if and only if the number of 2-simplices of $\T$ which contain $v$ is different from six.

To construct the Riemannian metric $g_\d$ we need to alter $g_s$ within small neighborhoods of simplices of $\T^{(n-2)}$, and then patch these metrics together using a smooth partition of unity.  Let us first carefully construct this collection of neighborhoods which we will denote $\O$.  The construction is recursive with $n-1$ steps, in the $l^{th}$ step we construct a collection of neighborhoods $\O_l$ with $\O_{l-1} \subset \O_l$ for $0 \leq l \leq n-2$.  Then $\O := \O_{n-2}$.

First choose $\d$ so that $0 < \d < \frac{1}{3}$, and for each vertex $v$ of $\T^{(n-2)}$ insert the open ball $b(v, \d)$ into $\O_0$.  Note that, since $\d < \frac{1}{3}$, each of these balls will be disjoint\footnote{If $\sigma$ is an equilateral simplex with volume $1$, then its edge lengths are an
 increasing function of its dimension $n$, hence the edge lengths will always be $\geq \frac{2}{\sqrt[4]{3}}>1$.}.  Next, let $e$ be an edge of $\T^{(n-2)}$.  Let $\bar{e} = e \setminus U$ where $U$ is the union of all of the sets contained in $\O_0$.  Since $e$ contains exactly two vertices, $\bar{e}$ is simply the interior of the edge $e$ with a segment of length $\d$ removed from each end.  Let $\O_1$ consist of all of the sets in $\O_0$, as well as a set of the form $b(\bar{e}, k_1 \d)$ for each edge $e \in \T^{(1)}$, where $k_1$ is a small positive constant.  For $k_1$ small enough, $b(\bar{e}, k_1 \d)$ will have nontrivial intersection with exactly two other members of $\O_1$, the open $\d$ balls about the vertices of $e$.

Defining the remaining collection of $\O_l$ recursively, let $\s \in \T^{(n-2)}$ be an $l$-dimensional simplex.  Let $\bar{\s} = \s \setminus U$ where $U$ is the union of all of the sets contained in $\O_{l-1}$.  Insert the open neighborhood $b(\bar{\s}, k_l \d)$ into $\O_l$ where $k_l < k_{l-1}$ is a small positive constant.  Also, let $\O_{l-1} \subset \O_l$.  For $k_l$ small enough, $b(\bar{\s}, k_l \d)$ will have nontrivial intersection with exactly the members of $\O_l$ corresponding to the faces of $\s$.  Letting $\O:= \O_{n-2}$ completes the construction.

Let $\mathcal{O} = \bigcup_{U \in \O} U$ and let $\mathcal{U} = b(M \setminus \mathcal{O}, k_{n-1} \d)$ be the open neighborhood of radius $k_{n-1} \d$ (for some small positive constant $k_{n-1} < k_{n-2}$) about the closed set $M \setminus \mathcal{O}$.  For $k_{n-1}$ small enough, $\mathcal{U}$ will not meet any faces of $\T$ with codimension greater than or equal to two.  

The collection $\O \cup \{ \mathcal{U} \}$ forms an open cover of $M$.  Since we are assuming that the smooth structure on $M$ is compatible with the triangulation $\T$, we may define smooth PL coordinates within each neighborhood of $\O$.  Within $\mathcal{U}$, we can define smooth PL coordinates interior to each $n$-simplex of $\T$.  Let $\{ \phi_i \}$ be a smooth partition of unity subordinate to $\O \cup \{ U \}$.

Interior to each open set $U \in \O$ we will define a smooth metric $g_U$.  Then the resulting Riemannian metric $g_\d$ will be defined by 

$$g_\d = \phi_{\mathcal{U}} g_s + \sum_{U \in \O} \phi_U g_U. $$

Let $U \in \O$ be arbitrary.  By our construction of $\O$, $U$ corresponds to some $l$-dimensional simplex $\s$.  We define the metric 
$g_U$ to simply be the pullback of the Euclidean metric under the smooth PL charts about $\s$ constructed above.  
We will express this metric in generalized cylindrical coordinates about $\s$ (which, if $l = 0$, would just be generalized spherical 
coordinates).  More specifically, if we denote the coordinates by $x_1, \ldots , x_l, r, \theta_1, ..., \theta_{n-l-1}$, then 
$$g_U = \sum_{i=1}^{l} dx_i^2 + dr^2 + r^2 d \theta_1^2 + \sum_{i=2}^{n-l-1} r^2 \sin^2(\theta_1)...\sin^2(\theta_{i-1}) d \theta_i^2   $$
where the last sum is void if $l=n-2$, and where the domains for the coordinates are:
$$ -\infty < x_i < \infty 	\qquad	  0 \leq r < \infty 	\qquad	 0 \leq \theta_j \leq \pi \; \; (j \neq n-l-1) 	\qquad	 0 \leq \theta_{n-l-1} < 2\pi $$

\subsection*{Choosing $\d$ so that $|\text{Vol}_{g_{\d}}(M) - \text{Vol}_{\T}(M)| < \e$}  To avoid overcomplicating this proof we will proceed as follows.  The simplest case of when $n=2$ will be carried out in full detail.  We then move on to the case when $n=3$ and work out in full detail the part that differs from the $n=2$ case.  The $n=3$ case captures the general behavior of the problem, and so it will be easy to explain from here how the result follows for general $n$.

\vskip 5pt

\noindent \underline{\bf Case $n=2$:} In this case, $\T^{(n-2)}$ is just the vertex set of $\T$ and thus $\O$ is a collection of disjoint subsets\footnote{This is one of the two main differences between the $n=2$ and the higher dimensional cases.  The other difference is that the triangulation automatically has high regularity. Vertex links are always $\S^1$ (hence the triangulation is PL), and points are always smoothly embedded in $M$.} of $M$.  Let $v \in \T^{(n-2)}$ be a vertex with corresponding open set $V \in \O$.  $g_\d$ differs from $g_s$ only in such neighborhoods $V$, and in $V$ the metric $g_\d$ has the form
$$ g_\d = \phi_V g_V + \phi_\U g_s. $$

Recall that we expressed $g_V$ in polar coordinates.  i.e., $g_V = dr^2 + r^2 d \theta^2$ where $0 \leq r < \d$ and $0 \leq \theta < 2 \pi$.  We can locally express $g_s$ in Cartesian coordinates by $g_s = dx_1^2 + dx_2^2$.  To compute $|\text{Vol}_{g_s}(V) - \text{Vol}_{g_\d}(V)|$, we need to convert $g_s$ to polar coordinates within $V$.  This is done as in any multivariable calculus course by setting $x_1 = \bar{r} \cos(\bar{\theta})$ and $x_2 = \bar{r} \sin(\bar{\theta})$.  But notice that the domains for these parameters are $0 \leq \bar{r} < \d$ and $0 \leq \bar{\theta} < \frac{\pi}{3} t_v$ where $t_v$ is the number of facets (in this case, triangles) containing $v$.  We then obtain that $0 \leq \frac{6}{t_v} \bar{\theta} < 2 \pi$, and so $\bar{r} = r$ and $\bar{\theta} = \frac{t_v}{6} \theta$.  Substituting these values and computing the differentials yields
$$ g_s = dr^2 + \left( \frac{t_v}{6} \right)^2 r^2 d \theta^2 $$
and thus
$$ g_\d = \phi_V g_V + \phi_\U g_s = dr^2 + \left( \phi_V + \left( \frac{t_v}{6} \right)^2 \phi_\U \right) r^2 d \theta^2. $$
We then compute the change in volume, $|\text{Vol}_{g_{\d}} (V) - \text{Vol}_{g_s} (V)|$, to equal
\begin{align*}
&=  \left| \int_{0}^{2 \pi} \int_{0}^{\d} \sqrt{\left( \phi_V + \left( \frac{t_v}{6} \right)^2 \phi_\U \right) r^2} dr d \theta - \int_{0}^{2 \pi} \int_{0}^{\d} \frac{t_v}{6} r dr d \theta \right| \\ 
&= \left| \int_{0}^{2 \pi} \int_{0}^{\d} \left( \sqrt{\left( \phi_V + \left( \frac{t_v}{6} \right)^2 \phi_\U \right) } - \frac{t_v}{6} \right) r dr d \theta \right| \\
&\leq \left| \sqrt{1 + \left(\frac{t_v}{6} \right)^2 } - \frac{t_v}{6} \right| \int_{0}^{2 \pi} \int_{0}^{\d} r dr d \theta \\
&= \left| \sqrt{1 + \left(\frac{t_v}{6} \right)^2 } - \frac{t_v}{6} \right| \pi \d^2 
\end{align*}
which approaches 0 as $\d$ approaches 0.  Since $M$ is compact there are only finitely many vertices, completing the proof when $n = 2$.

\vskip 5pt

\noindent \underline{\bf Case $n=3$:} When $n=3$, $\T^{(n-2)}$ is now the 1-skeleton of $\T$.  So the neighborhoods in $\O$ are no longer disjoint.  Let $v$ be a vertex of $\T$ and let $e$ be an edge containing $v$.  Denote their corresponding neighborhoods in $\O$ by $U_v$ and $U_e$, respectively.  The same argument as in the 2-dimensional case applies to the regions of $M$ where only $U_v$ or $U_e$ intersects $\U$.  What we need to show is that $\d$ can be chosen small enough so that $|\text{Vol}_{g_\d}(W) - \text{Vol}_s(W)| < \e$ with $W = \U \cap U_v \cap U_e$.  In what follows we adapt the notation $g_v := g_{U_v}$ and $g_e := g_{U_e}$.

In $W$, $g_v$ is written in spherical coordinates $g_v = (dr^v)^2 + (r^v)^2 (d \theta_1^v)^2 + (r^v)^2 $ $ \sin^2(\theta_1^v) (d \theta_2^v)^2$ (where the $v$ is emphasized in the coordinates to distinguish from the coordinates of $g_e$).  The domains of the parameters depend on $\d$ and $k_1$, but we will not get too caught up in those details here.  To compute $\text{Vol}_{g_\d}(W)$ we convert both $g_\U$ and $g_e$ to spherical coordinates.  In exactly the same way as when $n=2$ we have that $g_{\U} = (dr^v)^2 + C_v^2 (r^v)^2 (d \theta_1^v)^2 + C_e^2 (r^v)^2 \sin^2(\theta_1) (d \theta_2^v)^2$ where $C_v$ and $C_e$ are positive constants that depend on the number of facets of $\T$ that contain $v$ and $e$, respectively.\footnote{In fact, just as in the 2-dimensional case, $C_e = \frac{t_e}{6}$ where $t_e$ is the number of facets which contain $e$.  To compute $C_v$, one needs to compute the \emph{solid angle} $\varphi$ subtended by the three edges emanating from $v$.  Then $C_v = \frac{2 \pi}{\varphi t_v}$, where $t_v$ is the number of tetrahedra containing $v$.}

We need to convert $g_e$ into spherical coordinates within $W$.  Recall that $g_e$ is written in cylindrical coordinates within $U_e$, i.e. $g_e = dx_1^2 + (dr^e)^2 + (r^e)^2 (d \theta_1^e)^2  $.  Notice that by an orthogonal transformation within $U_v$, we may align the axis orthogonal to $\theta_2^v$ with the edge $e$ (see Figure \ref{firstfigure}).  With this choice of coordinates we see that $\theta_2^v = \theta_1^e$.  The domains for these two parameters may differ, but we can get an overestimate for the volume of $W$ by integrating over the larger of the two domains.  Also notice that
$$ x^1 = r^v \cos (C_v \theta_1^v) $$
$$ r^e = r^v \sin (C_v \theta_1^v) $$
where the constant $C_v$ is the same constant as in $g_\U$ and is due to the change in the domain of $\theta_1^v$ between $g_\d$ and $g_s$, exactly as in the 2-dimensional case where the constant was $\frac{t_v}{6}$.  Computing the differential then yields that $g_e = (dr^v)^2 + C_v^2 (r^v)^2 (d \theta_1^v)^2 + (r^v)^2 \sin^2(\theta_1^v) (d \theta_2^v)^2 $.  Thus the difference in volume, $|\text{Vol}_{g_{\d}} (V) - \text{Vol}_{g_s} (V)|$, equals

\begin{align*}
&= \left| \iiint_{W} \sqrt{\left( \phi_v + \phi_e + \phi_\U C_v^2 \right) \left( \phi_v + \phi_e C_e^2 + \phi_U C_v^2 \right) } (r^v)^2 \sin(\theta_1^v) dV \right. \\
&- \left. \iiint_{W} C_v C_e (r^v)^2 \sin(\theta_1^v) dV \right| \\
&\leq \left| \sqrt{\left( 1 + 1 + C_v^2 \right) \left( 1 + C_e^2 + C_v^2 \right) } - C_v C_e \right| \iiint_{W} (r^v)^2 \sin(\theta_1^v) dV \\
&= \left| \sqrt{\left( 1 + 1 + C_v^2 \right) \left( 1 + C_e^2 + C_v^2 \right) } - C_v C_e \right| \text{Vol}_{g_v}(W).
\end{align*}
It is clear that $\text{Vol}_{g_v}(W)$ goes to 0 as either $\d$ or $k_1$ approaches 0 (recall that $k_1$ was introduced in the construction of $\O$).  This completes the proof when $n=3$.


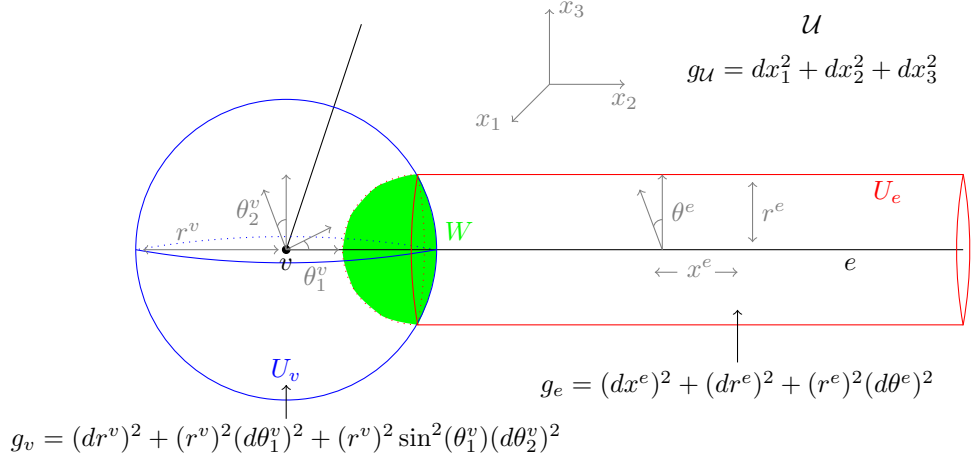
\begin{figure}[tb]
\begin{tikzpicture}

\path [fill=green] (1.75,1) to [out=175, in=205] (1.25,0.86603) to [out=205,in=235] (0.88397,0.5) to [out=235,in=265] (0.75,0) to [out=265,in=295] (0.88397,-0.5) to [out=295,in=325] (1.25,-0.86603) to [out=325,in=355] (1.75,-1) to [out= 65, in=295] (1.75,1);

\draw (0,0) -- (1,3);
\draw (0,0) -- (9,0);
\draw[blue] (0,0) circle [radius=2];
\draw[red] (1.75,1) -- (9,1);
\draw[red] (1.75,-1) -- (9,-1);
\draw[blue] (-2,0) arc [radius=11.5, start angle=260, end angle = 280];
\draw[blue, dotted] (2,0) arc [radius=11.5, start angle=80, end angle = 100];
\draw[red] (1.75,1) arc [radius=5.78, start angle=170, end angle = 190];
\draw[red, dotted] (1.75,-1) arc [radius=5.78, start angle=-10, end angle = 10];
\draw[red, dotted] (1.75,1) arc [radius=1, start angle=90, end angle = 270];
\draw[red] (9,1) arc [radius=5.78, start angle=170, end angle = 190];
\draw[red] (9,-1) arc [radius=5.78, start angle=-10, end angle = 10];

\draw[fill] (0,0) circle [radius=0.05];
\draw (0,0)node[below]{$v$};
\draw (7.5,0)node[below]{$e$};

\draw[red] (8,0.5)node[above]{$U_e$};
\draw[blue] (0,-1.6)node{$U_v$};
\draw (7,3)node{$\mathcal{U}$};
\draw (7,2.4)node{$g_{\mathcal{U}} = dx_1^2 + dx_2^2 + dx_3^2$};
\draw (6,-1.8)node{$g_e = (dx^e)^2 + (dr^e)^2 + (r^e)^2 (d \theta^e)^2$};
\draw [->] (6,-1.5) -- (6,-0.8);
\draw (0,-2.5)node{$g_v = (dr^v)^2 + (r^v)^2 (d \theta_1^v)^2 + (r^v)^2 \sin^2(\theta_1^v)(d \theta_2^v)^2$};
\draw [->] (0,-2.25) -- (0,-1.8);

\draw[gray] (5.5,0)node[below]{$x^e$};
\draw[gray] [->] (5.7,-0.3) -- (6,-0.3);
\draw[gray] [->] (5.2,-0.3) -- (4.9,-0.3);
\draw[gray] (6.2,0.5)node[right]{$r^e$};
\draw[gray] [<->] (6.2,0.1) -- (6.2,0.9);
\draw[gray] (5,0.5)node[right]{$\theta^e$};
\draw[gray] [->] (5,0) -- (5,1);
\draw[gray] [->] (5,0) -- (4.7,0.8);
\draw[gray] (5,0.4) arc [radius=0.12, start angle=90, end angle = 170];

\draw[gray] (-0.2,0.5)node[left]{$\theta^v_2$};
\draw[gray] [->] (0,0) -- (0,1);
\draw[gray] [->] (0,0) -- (-0.3,0.8);
\draw[gray] (0,0.4) arc [radius=0.12, start angle=90, end angle = 170];
\draw[gray] (0.4,-0.1)node[below]{$\theta^v_1$};
\draw[gray] [->] (0,0) -- (0.7,0);
\draw[gray] [->] (0,0) -- (0.6,0.3);
\draw[gray] (0.3,0) arc [radius=0.13, start angle=0, end angle = 60];
\draw[gray] (-1.3,0)node[above]{$r^v$};
\draw[gray] [<->] (-0.1,0) -- (-1.9,0);

\draw[gray] (3,1.7)node[left]{$x_1$};
\draw[gray] (4.5,2.2)node[below]{$x_2$};
\draw[gray] (3.5,3.2)node[right]{$x_3$};
\draw[gray] [<->] (4.5,2.2) -- (3.5,2.2) -- (3.5,3.2);
\draw[gray] [->] (3.5,2.2) -- (3,1.7);

\draw[green] (2.3,0)node[above]{$W$};

\end{tikzpicture}
\caption{The region $W$ is shaded in green, and the three different coordinate charts are written in gray.  Notice that, by an orthogonal change of coordinates within $U_v$, we have aligned the respective axes so that $\theta^e = \theta^v_2$.}
\label{firstfigure}
\end{figure}


\vskip 5pt

\noindent \underline{\bf Case $n>3$:} Once $n>3$ we must deal with intersections of three or more sets in $\O$.  But we always change coordinates to those of the simplex of least dimension and so we need only consider the intersection of two neighborhoods, one of which being that lowest dimensional simplex.  Denote these simplices by $\sigma$ and $\tau$ with dim($\s$)=$j$, dim($\tau$)=$l$, $j < l$.  Let $U_\s$ and $U_{\tau}$ denote their corresponding neighborhoods in $\O$, and denote their corresponding metrics by $g_\s$ and $g_{\tau}$.  All that we need to show is that we can change the coordinates in $\tau$ to coordinates in $\s$ in a way so that each component of the metric $g_{\tau}$ only differs from the corresponding component of $g_\s$ by a constant.  Notice that if $j > 0$ then we can project out the coordinates corresponding to the simplex $\s$ in both $U_\s$ and $U_{\tau}$.  So we may assume that $j=0$, $\s$ is a vertex, and thus the coordinates in $U_\s$ are generalized spherical coordinates
$$ g_\s = d\rho^2 + \rho^2 d \theta_1^2 + \sum_{i=2}^{n-1} \rho^2 \sin^2(\theta_1) ... \sin^2(\theta_{i-1}) d \theta_i^2 $$
where $0 \leq \rho < \infty$, $0 \leq \theta_i \leq \pi$ for $1 \leq i \leq n-2$, and $0 \leq \theta_{n-1} < 2 \pi$.  In $U_{\tau}$ the metric is written in generalized cylindrical coordinates by
$$ g_{\tau} = \sum_{i=1}^{l} dx_i^2 + d r^2 + r^2 d \bar{\theta_1}^2 + \sum_{i=2}^{n-l-1} r^2 \sin^2(\bar{\theta_1}) ... \sin^2(\bar{\theta}_{i-1}) d \bar{\theta_i}^2 $$
where $0 \leq r < \infty$, $0 \leq \bar{\theta_i} \leq \pi$ for $1 \leq i \leq n-l-2$, and $0 \leq \bar{\theta}_{n-l-1} < 2 \pi$.

Exactly as in the $n=3$ case, we can rotate the coordinates in $U_\s$ so that $\theta_{i +l} = \bar{\theta_i}$ for $1 \leq i \leq n-l-1$.  That still leaves $(n-1) - (n-l-1) = l$ directions in which the coordinates in $U_\s$ can be rotated.  Each of the parameters $\theta_1, \ldots, \theta_l$ in $U_\s$ measures an angle from a fixed positive axis.  We also rotate the coordinates in $U_\s$ so that the (positive) axis associated with $\theta_i$ corresponds to the coordinate $x_i$ in $U_{\tau}$.  We then have that $x_i = \rho \cos(C_i \theta_i)$ where the constant $C_i$ arrises from converting the domains of the associated variables in exactly the same way as the 2-dimensional case.  By projecting out $\rho$ along each of the dimensions $x_1, ..., x_l$ we can write $r$ as the product $r = \rho \sin(C_1 \theta_1) ... \sin(C_l \theta_l)$.  Thus, this change in coordinates only differs from the standard (Euclidean) change in coordinates by (possibly) multiplying each variable by a constant factor, and thus the coordinates of the corresponding metric only differ by a constant.  

\begin{rem}
In the construction of the metrics $g_\delta$, our metrics are changed by stretching or compressing in the radial directions 
about each simplex -- with the amount of distortion determined by the combinatorics of the link of the simplex. It follows that
there exists a constant $C_{\T}$ (depending solely on the triangulation $\T$) 
with the property that each of the identity maps from $(M, g_s)$ to $(M, g_\delta)$ are all $C_\T$-Lipschitz. 
\end{rem}

\subsection*{Lipschitz homotopies of closed paths to closed edge loops}

First, note that Theorem \ref{thm:riemannian-to-discrete} (2) is trivial for {\it any} constant if the closed path is null-homotopic.
Also, using the {\it Birkhoff curve shortening process}, a description of which can be found in \cite{Klingenberg}, 
we can homotope any closed path $\eta$ to a closed geodesic $\g$ in such a way that $\ell_g(\g) \leq \ell_g(\eta)$.
If we can then find an edge path $p$ satisfying Theorem \ref{thm:riemannian-to-discrete} (2) with respect to $\g$, then
	\begin{equation*}
	\ell_d(p) \leq \kappa_n \ell_g(\g) \leq \kappa_n \ell_g(\eta).
	\end{equation*}
So we may reduce to the case that $\g$ is a closed $g$-geodesic which is not null-homotopic.

The following Lemma and Corollary handle the case when $\g$ is a $g_s$-geodesic.

\begin{lem}\label{lem:g_s geodesics}
Let $\Delta$ be a Euclidean equilateral $n$-simplex with volume $1$, 
and let $\a$ be a straight line segment whose endpoints lie on the boundary of $\Delta$.
Then $\a$ can be homotoped, rel. endpoints, to a path $\a^{(1)} \subseteq \partial \Delta$ such that
	\begin{equation*}
	\ell(\a^{(1)}) \leq C_n \ell(\a)
	\end{equation*}
where the constant $C_n$ depends only on $n$.
\end{lem}

\begin{proof}
First, at the cost of a
very slight perturbation, we may assume that $\a$ misses the barycenter $B$ of the $n$-simplex $\Delta$.
Now, consider radial projection from $B$ to $\partial \Delta$.
Note that this map is {\it not} Lipschitz -- a small segment near the
barycenter will get stretched out to a long segment on the boundary.
But it {\it is} $C^\prime$-Lipschitz if one restricts to segments that are at least some
fixed (uniform) distance $D$ away from the barycenter.
Finally, if one
considers the segments that pass closer than $D$ to the barycenter, their
lengths are uniformly bounded below, while the length of their projected
images are uniformly bounded above.
So again, their is some constant $C^{\prime \prime}$
so that $\ell(\a^{(1)}) \leq C^{\prime \prime} \ell(\a)$.
Letting $C_n := \text{max}\{C^\prime, C^{\prime \prime} \}$ completes the proof.
\end{proof}

Let $\kappa^{\prime \prime} := C_n \cdot C_{n-1} \cdot \hdots \cdot C_2$.  
Define a {\it $g_s$-polygonal path} to be a path which is a $g_s$-geodesic (i.e., a straight line) when restricted to any simplex of $\T$.
Also, define the {\it support} of any path to be the collection of facets that it intersects.
Recursively applying Lemma \ref{lem:g_s geodesics} proves:

\begin{cor}\label{cor:g_s geodesics}
Let $\a$ be a closed $g_s$-polygonal path in $(M,\T)$.  
Then there exists an edge path $p$, freely homotopic to $\a$ and with the same support, such that
	\begin{equation}\label{eqn:Corollary}
	e_n \ell_d(p) \leq \kappa^{\prime \prime} \ell_{g_s} (\a)
	\end{equation}
where $e_n$ is the length of an edge of an equilateral $n$-simplex with volume $1$.
\end{cor}

\begin{proof}
Let $\a$ be a $g_s$-polygonal path.  
We inductively apply Lemma \ref{lem:g_s geodesics} to push our path from the $k$-skeleton to the $(k-1)$-skeleton.  
At each stage, we replace a path that is straight in each $k$-simplex by a path lying in the $(k-1)$-skeleton, 
and has the property that it is at most $C_k$ times the original length.
The path in the $(k-1)$-skeleton may no longer be straight on
each $(k-1)$-simplex, but one can straighten it on each of the simplices --
this only decreases the length of $p$ which does not effect inequality \eqref{eqn:Corollary} -- and then reapply the Lemma.
Note that points of $\a$ within any simplex stay within the boundary of that simplex throughout this procedure.  
So the support does not change throughout this process.  
Finally, we end up with a loop $p$ in the $1$-skeleton, homotopic to the original loop, and satisfying inequality \eqref{eqn:Corollary}.  
\end{proof}

This Corollary proves Theorem \ref{thm:riemannian-to-discrete} (2) for geodesics in the metric $g_s$ instead of $g$.
Intuitively, for $\d > 0$ very small, geodesics should not differ much in the metrics $g$ and $g_s$.  
But this takes a little work to show directly.  
So in what follows we take an arbitrary $g$-geodesic and reduce to the case of a $g_s$-polygonal path.  
The reader who believes that such a reduction is possible may skip ahead to Section \ref{section:applications}.

Note that $e_n \geq 1$ for all $n \geq 1$, and therefore the above Corollary
gives that $\ell_d(p) \leq \kappa^{\prime \prime} \ell_{g_s}(\a)$.
Define $\kappa^\prime := 3 \kappa^{\prime \prime}$.  
Then, using the notation of the above Corollary, we have that
	\begin{equation}\label{eqn:kp}
	\ell_d(p) \leq e_n \ell_d(p) < 3e_n \ell_d(p) \leq \kp \ell_{g_s}(\a).
	\end{equation}
	
The following Lemma allows us to apply the preceeding Corollary to closed $g$-geodesics 
which do not intersect any simplex of $\T$ ``too many times''.
	
\begin{lem}\label{lem:g-polygonal paths}
Let $\g$ denote a closed $g$-polygonal path in $(M, \T)$, and let $K > 0$ be some fixed constant.  Suppose that:
	\begin{enumerate}
	\item  $\g$ is not null-homotopic
	\item  For all $\s \in \T$, $\g \cap \s$ consists of at most $K$ connected components.
	\end{enumerate}
Then there exists a closed $g_s$-polygonal path $\a$, freely homotopic to $\g$, such that
	\begin{equation*}
	\frac{1}{2} \ell_{g_s}(\a) \leq \ell_{g}(\g)
	\end{equation*}
\end{lem}

In Lemma \ref{lem:g-polygonal paths} a {\it $g$-polygonal path} is a path which is a $g$-geodesic when restricted
to any simplex of $\T$.

\begin{proof}
Let $\g$ be a closed $g$-polygonal path.  
If $\g$ does not intersect the $\d$-neighborhood of $\T^{(n-2)}$ then $\g$ is a polygonal path in the $g_s$ metric and we are done.  
So assume that $\g$ intersects the $\d$-neighborhood of $\T^{(n-2)}$, denoted $b_{\d}$.  
Let $\b_1, \hdots, \b_k$ denote the connected components of $\g \cap b_\d$.

Consider one of these components $\b_i$.  
Let $\t_1, \hdots, \t_l$ denote the simplices of $\T$ with codimension 2 or greater for which $\b_i$ intersects the corresponding warped neighborhood $U_{\t_i}$.  

Recall that, when defining the metric $g$, we altered the $g_s$ metric about the $k_i \d$ neighborhood of each $i$-dimensional
face (and $i \leq n-2$).  
Then, via a compactness argument, we can systematically choose $k_1, \hdots, k_{n-2}$ small enough 
so that $b_{g}(x, 2\d) \setminus b_\d$ is non-empty for each point $x$ in the $(n-2)$-skeleton of $\T$.
The point here is that we can choose the region $\U$ in which we are altering the $g_s$-metric small enough so that the
open $2\d$ ball about any point of $M$, measured in the $g$-metric, contains points outside of $\U$.


Now, we want to remove the interior of $\b_i$ and replace it with a $g_s$-polygonal path, denoted $\a_i$, 
between its endpoints which stays within a $4 \d$ neighborhood of $\b_i$.  
We can always find such a path $\a_i$ so that
	\begin{equation}\label{eqn:polygonal approximation}
	\ell_{g_s}(\a_i) \leq \ell_g(\b_i) + 8 \d | \T_i |
	\end{equation}
where $\T_i $ denotes the collection of all simplices of $\T$ which contain any of $\t_1, \hdots, \t_l$.
To see this, just subdivide $\b_i$ where it intersects different simplices of $\T$.  
Sequentially approximate each of these points by a point outside of $b_\d$ at a distance of at most $4 \d$ away 
(whose existence is guaranteed by the preceding paragraph), 
and then connect each of these points by a $g_s$-polygonal path.
Equation \eqref{eqn:polygonal approximation} then follows from repeated application of the triangle inequality.
Note that inequality \eqref{eqn:polygonal approximation} is a very crude estimate.  
In general, one would not need anywhere near $|\T_i|$ polygonal pieces in any such approximation.

Now consider the polygonal path, denoted by $\a$, obtained by replacing each $\b_i$ 
with its corresponding polygonal approximation $\a_i$.  
Then one sees immediately that
	\begin{equation}\label{eqn:length of polygonal approximation}
	\ell_{g_s}(\a) \leq \ell_g(\g) + 8\d \sum_{i=1}^k |\T_i|.
	\end{equation}
Due to assumption (2) we have that
	\begin{equation}\label{eqn:bound on simplices}
	\sum_{i=1}^k |\T_i| \leq \mu K ||\T|| 
	\end{equation}
where $||\T||$ denotes the total number of simplices contained in $\T$, and $\mu$ denotes the maximal degree of any simplex of $\T$
(i.e., $\ds{ \mu = \text{max} \{ ||St(\s)|| : \, \s \in \T \} }$ where $St(\s)$ denotes the {\it closed star} of the simplex $\s$).

Now, choose 
	\begin{equation}\label{eqn:delta choice}
	\d < \frac{1}{16 \mu K ||\T||} \text{sys}_{g_s}(M)
	\end{equation}
where sys$_{g_s}(M)$ denotes the systole of $M$ with respect to the metric $g_s$.  

Combining inequalities \eqref{eqn:length of polygonal approximation}, \eqref{eqn:bound on simplices}, and \eqref{eqn:delta choice} yields:
	\begin{align*}
	\ell_g(\g) &\geq \ell_{g_s}(\a) - 8 \d \sum_{i=1}^k |\T_i| \\
	&\geq \ell_{g_s}(\a) - 8\d \mu K ||\T||  \\
	&\geq \ell_{g_s}(\a) - \frac{1}{2} \text{sys}_{g_s}(M)  \\
	&\geq \frac{1}{2} \ell_{g_s}(\a)
	\end{align*}
and where, for the last inequality, it is necessary that $\a$ is not null-homotopic.
\end{proof}

We can now complete the proof of Theorem \ref{thm:riemannian-to-discrete} with
$\kappa = 2 \kp$ for $g$-geodesics which satisfy the conditions of Lemma \ref{lem:g-polygonal paths}.  
Let $\g$ be such a $g$-geodesic.  
Let $\a$ be the $g_s$-polygonal path guaranteed by Lemma \ref{lem:g-polygonal paths},
and let $p$ be the edge path from the Corollary corresponding to $\a$.  Then
	\begin{equation*}
	\ell_d(p) \leq \kp \ell_{g_s}(\a) \leq 2 \kp \ell_g(\g) = \kappa \ell_g(\g).
	\end{equation*}

In order to complete the proof of Theorem \ref{thm:riemannian-to-discrete}, we need to 
fix some $K>0$ and reduce to the case of $g$-polygonal paths which intersect each simplex of $\T$ at most $K$ times.

In order to define $K$, let us first define 
	\begin{equation}\label{eqn:D}
	D := \frac{e_n}{\kp} \quad	\Longrightarrow	\quad	e_n = D \kp
	\end{equation}
where $e_n$ and $\kp$ are as in equation \eqref{eqn:kp}.
Note that $D$ depends only on $n$.

We now define $K$ as follows.  
Cover the $(n-2)$-skeleton $\T^{(n-2)}$ of $\T$ with a finite number of open $\frac{1}{8} D$-balls (in the $g_s$ metric).
Then extend this cover to an open cover of $\T^{(n-1)}$ by open $\frac{1}{8} D$-balls; 
denote by $\{p_1, \ldots , p_N\}$ the points where these balls are centered.  
$K$ is then the maximal number of open sets in this covering required to cover the boundary of any simplex of $\T$.
Note that, since the Riemannian manifold $(M,g)$ converges to the geodesic metric space $(M,g_s)$ in the Gromov-Hausdorff sense 
(as $\d$ approaches $0$), we can choose $\d > 0$ sufficiently small so that the collection of $g$-metric open $\frac{1}{8} D$-balls, centered at the same collection of points $\{p_1, \ldots ,p_N\}$, also forms a 
cover of $\T^{(n-1)}$ -- call this open cover $\mathcal U$.
Also note that there is no ambiguity with equation \eqref{eqn:delta choice}, as $K$ is fixed and then we choose $\d$. 

\begin{rem}
Let $U, V \in \mathcal U$ be such that $U \cap V \neq \emptyset$.  
Then by the above construction, $\text{diam}_g(U \cup V) \leq \frac{1}{2} D$.  
Let $\s \in \T$ be a simplex that intersects both $U$ and $V$.  
Let $x, y \in \s \cap U \cap V$, and let $\g$ be a $g$-geodesic joining $x$ and $y$ (so, in particular, $\ell_g(\g) \leq \frac{1}{2} D$).
A priori, $\g$ could weave in and out of $\s$.  
But by choosing $\d$ small enough, we can ensure that any such points $x$ and $y$ can be connected by a path of length less than $D$ which is a geodesic of $g$ restricted to $\s$.  
\end{rem}

\begin{proof}[Proof of Theorem \ref{thm:riemannian-to-discrete}]
Let $\g$ denote a closed $g$-polygonal path, and let $K$ be as above.
Suppose that $\g$ is not null-homotopic, but does not satisfy condition (2) of Lemma \ref{lem:g-polygonal paths}.
Let us assume that $\s \in \T$ is the only simplex for which $\g \cap \s$ consists of more than $K$ connected components, and 
that $\g \cap \s$ has exactly $K+1$ components. 
The procedure described below can be iterated (see Step 4 below) to deal with multiple simplices and/or for a greater intersection number with any simplex.

Fix a base point and orientation of $\mathbb{S}^1$.  
Using this orientation, each component of $\g \cap \s$ has an ``entrance point'' $x_i$ and an ``exit point'' $y_i$.
Let $x_1, \hdots, x_{K+1}$ denote the $K+1$ entrance points, and let $y_1, \hdots, y_{K+1}$ denote the $K+1$ exit points.
By the definition of $K$, there must exist two entrance points $x_i$ and $x_j$ such that
	\begin{equation}\label{eqn:entrance distance}
	d_g (x_i,x_j) < D 	\quad		\overset{\text{Eqn} \eqref{eqn:D}}{\Longrightarrow}		\quad		\kp d_g (x_i,x_j) < e_n
	\end{equation}
Let $y_i$ and $y_j$ denote the corresponding exit points, and assume that $i < j$.

Remove the segments $(x_i,y_i)$ and $(x_j,y_j)$ from $\g$, 
and insert $g$-geodesics $(x_i,y_j)$ and $(x_j,y_i)$ interior to $\s$.  
By the above Remark, we know that these two segments have length less than $D$.
Denote the closed $g$-polygonal path containing $(x_i,y_j)$ by $\g_1$, and the other by $\g_2$.
We may assume that neither path is null-homotopic, for otherwise we could have homotoped the original path interior to $\s$
and reduced the number of components of $\g \cap \s$.  
So by Lemma \ref{lem:g-polygonal paths}, the Corollary, and equation \eqref{eqn:kp} there exist
closed edge paths $p_1$ and $p_2$ freely homotopic to $\g_1$ and $\g_2$ such that
	\begin{equation}\label{eqn:basic inequalities}
	3e_n \ell_d(p_1) \leq \kp \ell_g(\g_1) \quad \text{and} \quad 3e_n \ell_d(p_2) \leq \kp \ell_g(\g_2).
	\end{equation}

Now under the free homotopy of $\gamma_1$ into the edge path $p_1$, the points $x_i$, $y_j$
find themselves lying on the $1$-skeleton of $\sigma$. Let $\vxi$ and $\vyj$ be the vertices of
$\sigma$ which are closest to the image of $x_i, y_j$ respectively. Similarly, the free homotopy of $\gamma_2$ into the edge
path $p_2$ moves $x_j, y_i$ into the $1$-skeleton of $\sigma$, and we let $\vxj, \vyi$ be the vertices of $\sigma$ closest
to these image points.
By possibly shortening the paths if necessary, we may assume that the edge $\overline{\vxi \vyj} \in p_1$ 
and the edge $\overline{\vxj \vyi} \in p_2$. Note that it is entirely possible that one (or both) of these edges is
degenerate.

We want to simultaneously 
	\begin{itemize}
	\item use $p_1$ and $p_2$ to construct a closed path $p$ that is freely homotopic to $\g$.
	\item Reconstruct $\g$ from $\g_1$ and $\g_2$.
	\item Preserve the key inequality $\ds{ \ell_d(p) \leq \kp \ell_g(\g) }$.
	\item Do all of this in a manner which can be iterated.
	\end{itemize}
We will do this in four steps. In steps 2 and 3, we will need to assume that the edges  $\overline{\vxi \vyj},
\overline{\vxj \vyi}$ are non-degenerate. However, in step 4,  we will explain how to allow for degenerate edges.

\vskip 10pt

\begin{figure}\label{Step 1}
\begin{tikzpicture}[scale=0.9]

\draw (0,2)node{$\underline{p^{\prime \prime}}$};
\draw[blue] (-1,1) -- (-1,-1);
\draw (1,1) -- (1,-1);
\draw[red] (-1,1) -- (1,1);
\draw[red] (-1,-1) -- (1,-1);
\draw[fill=black!] (1,1) circle (0.3ex);
\draw (1,1)node[above]{$v_j$};
\draw[fill=black!] (1,-1) circle (0.3ex);
\draw (1,-1)node[below]{$w_i$};
\draw[fill=black!] (-1,-1) circle (0.3ex);
\draw (-1,-1)node[below]{$v_i$};
\draw[fill=black!] (-1,1) circle (0.3ex);
\draw (-1,1)node[above]{$w_j$};
\draw[dashed] (-1,1) to [out=180,in=90] (-3,0);
\draw[dashed] (-3,0) to [out=270,in=180] (-1,-1);
\draw(-3,0)node[right]{$p_1$};
\draw[dashed] (1,-1) to [out=0,in=270] (3,0);
\draw[dashed] (1,1) to [out=0,in=90] (3,0);
\draw(3,0)node[left]{$p_2$};

\draw (8,2)node{$\underline{\g^{\prime \prime}}$};
\draw[dashed] (7,1) to [out=180,in=90] (5,0);
\draw[dashed] (7,-1) to [out=180,in=270] (5,0);
\draw[fill=black!] (7,1) circle (0.3ex);
\draw (7,1)node[above]{$y_j$};
\draw[fill=black!] (7,-1) circle (0.3ex);
\draw (7,-1)node[below]{$x_i$};
\draw (7,1) to [out=280,in=80] (7,-1);
\draw(5,0)node[right]{$\g_1$};

\draw[dashed] (9,1) to [out=0,in=90] (11,0);
\draw[dashed] (9,-1) to [out=0,in=270] (11,0);
\draw[fill=black!] (9,1) circle (0.3ex);
\draw (9,1)node[above]{$x_j$};
\draw[fill=black!] (9,-1) circle (0.3ex);
\draw (9,-1)node[below]{$y_i$};
\draw[blue] (9,1) to [out=260,in=100] (9,-1);
\draw(11,0)node[left]{$\g_2$};

\end{tikzpicture}
\caption{Schematic picture for Step 1.  In red is what was added in this step, and in blue is what will be deleted in Step 2.
Note that this picture is just to keep track of what is going on.  It is not geometrically accurate, as $x_i$ is close to $x_j$ in the proof.  }
\end{figure}
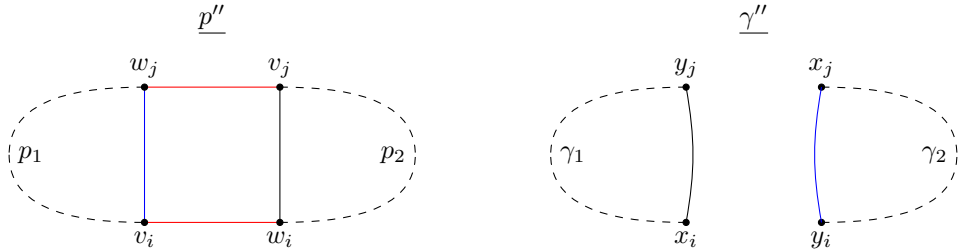

\noindent {\bf Step 1:  Append edges $\overline{\vxi \vyi}$ and $\overline{\vxj \vyj}$ between $p_1$ and $p_2$.}

\vskip 5pt

Let $p^{\prime \prime} := p_1 \cup p_2 \cup \overline{\vxi \vyi} \cup \overline{\vxj \vyj}$, and let $\g^{\prime \prime} := \g_1 \cup \g_2$ (see Figure 4  below for a schematic illustration). Now let us check that the analogue of the key inequality still holds. This is a consequence of the following series of inequalities:
	\begin{align*}
	2e_n \ell_d(p^{\prime \prime}) &\leq 2e_n (\ell_d(p_1) + \ell_d(p_2) + 2)  \\
	&= 2e_n (\ell_d(p_1) + 1) + 2e_n (\ell_d(p_2)+1)  \\
	&\leq 2e_n \left( \ell_d(p_1) + \frac{1}{3} \ell_d(p_1) \right) + 2e_n \left( \ell_d(p_2) + \frac{1}{3} \ell_d(p_2) \right)  \\
	&< 3e_n \ell_d(p_1) + 3e_n \ell_d(p_2)  \\
	&\leq \kp \ell_g(\g_1) + \kp \ell_g(\g_2) \\
	& =  \kp \ell_g(\g^{\prime \prime}).
	\end{align*}
	
The second inequality holds because both $p_1$ and $p_2$ are not null-homotopic, so must consist of
at least three edges. The last inequality is due to the ``3'' present in equation \eqref{eqn:basic inequalities}. Note that,
in the event that one of the edges $\overline{\vxi \vyi}, \overline{\vxj \vyj}$ is degenerate, this series of inequalities 
still holds (the only effect is that the first inequality in the chain becomes strict).

\vskip 10pt

\noindent {\bf Step 2:  Insert $(x_i,y_i)$ and remove $(x_j,y_i)$ from $\g^{\prime \prime}$, 
and remove the edge $\overline{\vxi \vyj}$ from $p^{\prime \prime}$.}

\vskip 5pt

Let $p^\prime$ and $\g^\prime$ denote the new sets created from the above procedures (see Figure 5 below).  
Notice that, by the triangle inequality and equation \eqref{eqn:entrance distance} we get
	\begin{equation*}
	d_g(x_j,y_i) \leq d_g(x_i,y_i) + d_g(x_i,x_j)  
	\end{equation*}
	\begin{equation}
	\Longrightarrow \quad -D \leq -d_g(x_i,x_j) \leq d_g(x_i,y_i) - d_g (x_j,y_i).  \label{eqn:3}
	\end{equation}
Then by equations \eqref{eqn:3} and \eqref{eqn:D} we obtain
	\begin{equation}\label{eqn:edge deletion}
	2 e_n \ell_d(p^\prime) = 2 e_n \ell_d(p^{\prime \prime}) -2e_n
	\end{equation}
	\begin{equation*}
	\kp \ell(\g^\prime) = \kp [\ell(\g^{\prime \prime}) + d_g(x_i,y_i) - d_g(x_j,y_i)] 
	\geq \kp [\ell(\g^{\prime \prime}) - D] = \kp \ell(\g^{\prime \prime}) - e_n
	\end{equation*}
and so
	\begin{equation}\label{eqn:4}
	2e_n \ell_d(p^\prime) = 2e_n \ell_d(p^{\prime \prime}) - 2e_n \leq \kp \ell_g(\g^{\prime \prime}) - e_n \leq \kp \ell_g(\g^\prime).
	\end{equation}

Notice that here, it is important that the edge $\overline{\vxi \vyj}$ is non-degenerate. If
it were degenerate, equation (\ref{eqn:edge deletion}) would not hold, and thus neither would                                                                                                                                                                                                                                                                                                                                                                              (\ref{eqn:4}).

\begin{figure}\label{Step 2}
\begin{tikzpicture}[scale=0.9]

\draw (0,2)node{$\underline{p^{\prime}}$};
\draw[blue] (1,1) -- (1,-1);
\draw (-1,1) -- (1,1);
\draw (-1,-1) -- (1,-1);
\draw[fill=black!] (1,1) circle (0.3ex);
\draw (1,1)node[above]{$v_j$};
\draw[fill=black!] (1,-1) circle (0.3ex);
\draw (1,-1)node[below]{$w_i$};
\draw[fill=black!] (-1,-1) circle (0.3ex);
\draw (-1,-1)node[below]{$v_i$};
\draw[fill=black!] (-1,1) circle (0.3ex);
\draw (-1,1)node[above]{$w_j$};
\draw[dashed] (-1,1) to [out=180,in=90] (-3,0);
\draw[dashed] (-3,0) to [out=270,in=180] (-1,-1);
\draw(-3,0)node[right]{$p_1$};
\draw[dashed] (1,-1) to [out=0,in=270] (3,0);
\draw[dashed] (1,1) to [out=0,in=90] (3,0);
\draw(3,0)node[left]{$p_2$};

\draw (8,2)node{$\underline{\g^{\prime}}$};
\draw[dashed] (7,1) to [out=180,in=90] (5,0);
\draw[dashed] (7,-1) to [out=180,in=270] (5,0);
\draw[fill=black!] (7,1) circle (0.3ex);
\draw (7,1)node[above]{$y_j$};
\draw[fill=black!] (7,-1) circle (0.3ex);
\draw (7,-1)node[below]{$x_i$};
\draw[blue] (7,1) to [out=280,in=80] (7,-1);
\draw(5,0)node[right]{$\g_1$};

\draw[dashed] (9,1) to [out=0,in=90] (11,0);
\draw[dashed] (9,-1) to [out=0,in=270] (11,0);
\draw[fill=black!] (9,1) circle (0.3ex);
\draw (9,1)node[above]{$x_j$};
\draw[fill=black!] (9,-1) circle (0.3ex);
\draw (9,-1)node[below]{$y_i$};
\draw[red] (7,-1) to [out=10,in=170] (9,-1);
\draw(11,0)node[left]{$\g_2$};

\end{tikzpicture}
\caption{Schematic picture for Step 2.  Again, what is new is in red, and what will be removed in Step 3 is in blue.}
\end{figure}
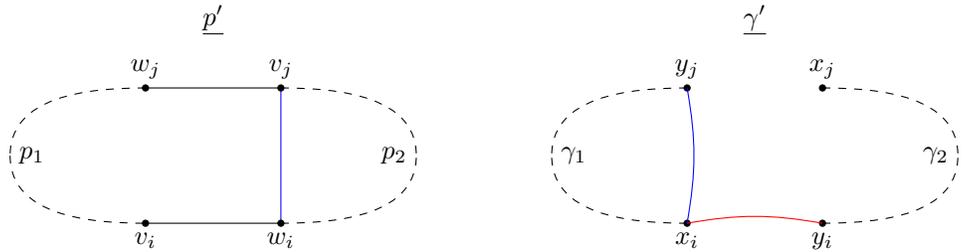

\vskip 10pt

\noindent {\bf Step 3:  Insert $(x_j,y_j)$ and remove $(x_i,y_j)$ from $\g^{\prime}$, 
and remove the edge $\overline{\vxj \vyi}$ from $p^{\prime}$.}

\vskip 5pt

The first two operations will return our original path $\g$, and the second will provide a closed edge path $p$.
The path $p$ is freely homotopic to $\g$ since all of our operations occurred within the closed simplex $\s$.
By the triangle inequality, we have that
	\begin{equation*}
	d_g(x_i,y_j) \leq d_g(x_j,y_j) + d_g(x_i,x_j)
	\end{equation*}
and the exact same argument as for equation \eqref{eqn:4} proves that
	\begin{equation}\label{eqn:5}
	\ell_d(p) < 2e_n \ell_d(p) \leq \kp \ell(\g)
	\end{equation}

Again, in this step, one needs the edge $\overline{\vxj \vyi}$ to be non-degenerate.
	
\vskip 10pt

\noindent {\bf Step 4:  A few remarks to ensure that this process iterates.}

\vskip 5pt

The ``3'' in equation \eqref{eqn:kp} means that, at each vertex of both $p_1$ and $p_2$, 
we can append two edges to obtain new sets $p_1^\prime$ and $p_2^\prime$ which still satisfy that
$\ell_d(p_1^\prime) \leq \kp \ell_g(\g_1)$ and $\ell_d(p_2^\prime) \leq \kp \ell_g(\g_2)$
	
The reason that we need this multiple of three
is because any of the vertices $\vxi$, $\vxj$, $\vyi$, and/or $\vyj$ could
be the same. As already mentioned in
Steps 2 and 3, the proofs for inequalities \eqref{eqn:4} and \eqref{eqn:5} do not hold 
without inequality \eqref{eqn:edge deletion}. But for this inequality to hold, we must have an edge to delete.  
This edge may not exist if $\vxi = \vyj$ and/or $\vxj = \vyi$, and
so we may need an additional edge built in for these steps
as sort of an ``extraneous edge'' that we can delete in order to preserve inequality \eqref{eqn:edge deletion}.
	
Since the ``3'' in equation \eqref{eqn:kp} is multiplicative, we can glue in these two additional edges
at {\it every} vertex.  Thus, we always have these edges available to us wherever we cut $\g$ into two closed 
curves $\g_1$ and $\g_2$.

\end{proof}

\section{Filling triangulated surfaces}\label{section:applications}

Recall that, given a closed triangulated $n$-dimensional manifold $(M, \T_M)$, a \emph{filling} of $M$ is a triangulated $(n+1)$-dimensional manifold $(N,\T_N)$ with $\partial N = M$ and $\T_N |_{\partial N} = \T_M$.  A basic question is the following.  Given a triangulated manifold $(M, \T_M)$, does such a filling exist and, if so, can you bound $|\T_N|$, the number of facets of such a filling?  Theorem \ref{thm:riemannian-to-discrete} leads to the following two solutions to this question in the case when $n=2$.

\begin{thm}\label{lineargenus}
  Let $(M, \T_M)$ be a triangulated surface of genus $\leq g$. 
  Then there exists a filling $(N,\T_N)$ satisfying that
  \begin{equation*}
    |\T_N| \leq C_g |\T_M|,
  \end{equation*}
  where $C_g$ depends only on $g$, and not on the particular surface or triangulation.
\end{thm}

\begin{thm}\label{quadratic}
  Let $(M,\T_M)$ be a triangulated surface. Then there exists $(N, \T_N)$, a filling of $M$, so that
  \begin{equation*}
    |\T_N| \leq C |\T_M| \left( \log |\T_M| \right)^2,
  \end{equation*}
  where $C$ does not depend on the particular surface or triangulation.
\end{thm}

The proofs of both of these theorems are very similar. We first combine Theorem \ref{thm:riemannian-to-discrete} with results of Gromov in \cite{Gromov1} and \cite{Gromov2} to bound the discrete systole of $(M,\T_M)$ by a factor of the combinatorial volume of $(M, \T_M)$.  Then the main idea is to apply a ``cut-and-cone" procedure.  We begin this procedure by 
cutting the surface along a short homologically nontrivial edge loop. This will 
yield a surface of smaller genus with two boundary components. We then cone off 
the boundary components to get a surface of genus one less than the 
original surface (See Figure \ref{cac}). We iterate this procedure until the 
surface is a 2-sphere, in which case we perform a modified coning-off procedure to 
get a triangulated 3-ball. By gluing the 3-ball along all of the cuts in the 
reverse order, we obtain a triangulated 3-manifold with the desired properties.

\begin{rem}
The argument for our proofs ``builds'' the bounding $3$-manifold from the triangulation on $\Sigma$. One might wonder whether
this is really necessary. Indeed, if one takes the genus $g$ handlebody $H_g$ embedded in $\mathbb R^3$, any triangulation of the
boundary surface $\Sigma_g$ can be extended in to a triangulation of $H_g$. The following Lemma shows that $H_g$ is in general not the best bounding surface for $\Sigma_g$.
\end{rem}

\begin{lem}\label{bad-triangulations}
One can construct a sequence of triangulations $\mathcal T_i$ of the boundary $\Sigma_g$ with a {\it fixed} number of triangles $|\mathcal T_i|\leq 24g$, and with the property that any extension to a triangulation $\widehat{\mathcal T_i}$ of $H_g$ satisfies $|\widehat{\mathcal T_i}|\to \infty$.
\end{lem}

\begin{proof}
By way of contradiction, let us assume that there is a universal upper bound $|\widehat{\mathcal T_i}|\leq K$. Since there are only finitely many $3$-complexes that can be built from $K$ tetrahedra, one can look at the finitely many such complexes that are homeomorphic to $H_g$. For each of these, there are finitely many embedded curves in the $1$-skeleton of the boundary, and for each of these curves, we can look at the image in $\pi_1(H_g)\cong F_g$. Thus, if one has such a universal upper bound, we get a finite collection of elements in $\pi_1(H_g)$ which carry all possible embedded curves in the $1$-skeleton of the boundary. But now starting with a triangulation $\mathcal T$ of $\Sigma_g$, one can push forward $\mathcal T$ under powers of a suitable Dehn twist (chosen along a curve which is non-trivial in $\pi_1(H_g)$). It is easy to see that the resulting sequence of triangulations on $\Sigma_g$ has embedded curves in the $1$-skeleton whose image in $\pi_1(H_g)$ form an unbounded set. 
\end{proof}

Of course, what is underlying the previous example is the fact that the natural homomorphism $MCG(H_g) \rightarrow MCG(\Sigma_g)$ has infinite index (where $MCG$ denotes the mapping class group -- the group of homotopy classes of homeomorphisms of the manifold). A similar argument can be used to give higher dimensional examples. Lemma \ref{bad-triangulations} shows that the choice of a good bounding $3$-manifold must depend on the initial triangulation of $\Sigma_g$.

\begin{rem} Some variations of our notion of filling function have previously been considered in the literature. For example,
Hass, Snoeying, and W. Thurston \cite{HST} have considered unknotted polygonal curves in $\mathbb R^3$, and studied the minimal number of triangles in a PL spanning disk for the curve. They give an exponential lower bound for the corresponding filling function, with an upper bound subsequently obtained by Hass, Lagarias, and W. Thurston \cite{HLT}. The corresponding question for knotted polygonal curves bounding PL surfaces was considered by Hass and Lagarias \cite{HL}. In a somewhat different direction, Costantino and D. Thurston \cite{CT} considered a similar question for $3$-manifolds -- but did not require the optimal triangulation on the bounding $4$-manifold to restrict to the original triangulation on the $3$-manifold.
\end{rem}

\subsection*{Discrete analogues of Riemannian systolic inequalities}  

We first need the following Lemma:

\begin{lem}\label{sys2comblemma}
  Let $(M, \T)$ be a closed triangulated $n$-dimensional manifold and let 
  $P_1,\dots, P_N$ be free-homotopy-invariant properties a loop in $M$ can 
  satisfy. Suppose that, for each $\epsilon > 0$, there is a closed geodesic 
  $\gamma_\epsilon$ on the Riemannian manifold $(M, g)$ (where $g$ is the metric from Theorem \ref{thm:riemannian-to-discrete}) so that $\gamma_\epsilon$ satisfies properties $P_1,\dots,P_N$ 
  and
  \begin{equation}\label{sys2comblemmahyp}
    \ell_{g}(\gamma_\epsilon) \leq C \sqrt{\text{Vol}_{g}(M)}.
  \end{equation}
  Then there is an edge loop $p$ on $M$ so that $p$ satisfies 
  properties $P_1,\dots,P_N$ and
  \begin{equation}\label{sys2comblemmaconc}
    \ell_{\T}(p) \leq  \kappa_n C \sqrt{\text{Vol}_{\T}(M)}.
  \end{equation}
\end{lem}

\begin{proof}
  Let $\epsilon > 0$ and let $\gamma_\epsilon$ be a noncontractible closed 
  geodesic.  By Theorem \ref{thm:riemannian-to-discrete}, 
  there exists an edge loop $p_\epsilon$ freely homotopic to $\gamma_\epsilon$ such that
  \begin{equation}\label{homedgeloop}
    \ell_{\T}(p_\epsilon) \leq \kappa_n \ell_{g}(\gamma_\epsilon).
  \end{equation}
  Combining inequalities \eqref{sys2comblemmahyp} and \eqref{homedgeloop}, we 
  have that
  \begin{equation}\label{comblooplength}
    \ell_{\T}(p_\epsilon) \leq \kappa_n \ell_{g}(\gamma_\epsilon) 
    \leq \kappa_n C \sqrt{Vol_{g}(M)} \leq \kappa_n C \sqrt{\text{Vol}_{\T}(M)+ \epsilon}.
  \end{equation}
  Thus we have a collection of edge loops $p_\epsilon$ in $M$ that 
  satisfy \eqref{comblooplength}, all of which have properties $P_1,\dots,P_N$. 
  Since the set of edge loops in $M$ is finite, there must be an 
  edge loop $p$ with properties $P_1,\dots,P_N$ whose length satisfies 
  \eqref{sys2comblemmaconc}.
\end{proof}

\begin{figure}
\includegraphics[scale=0.6]{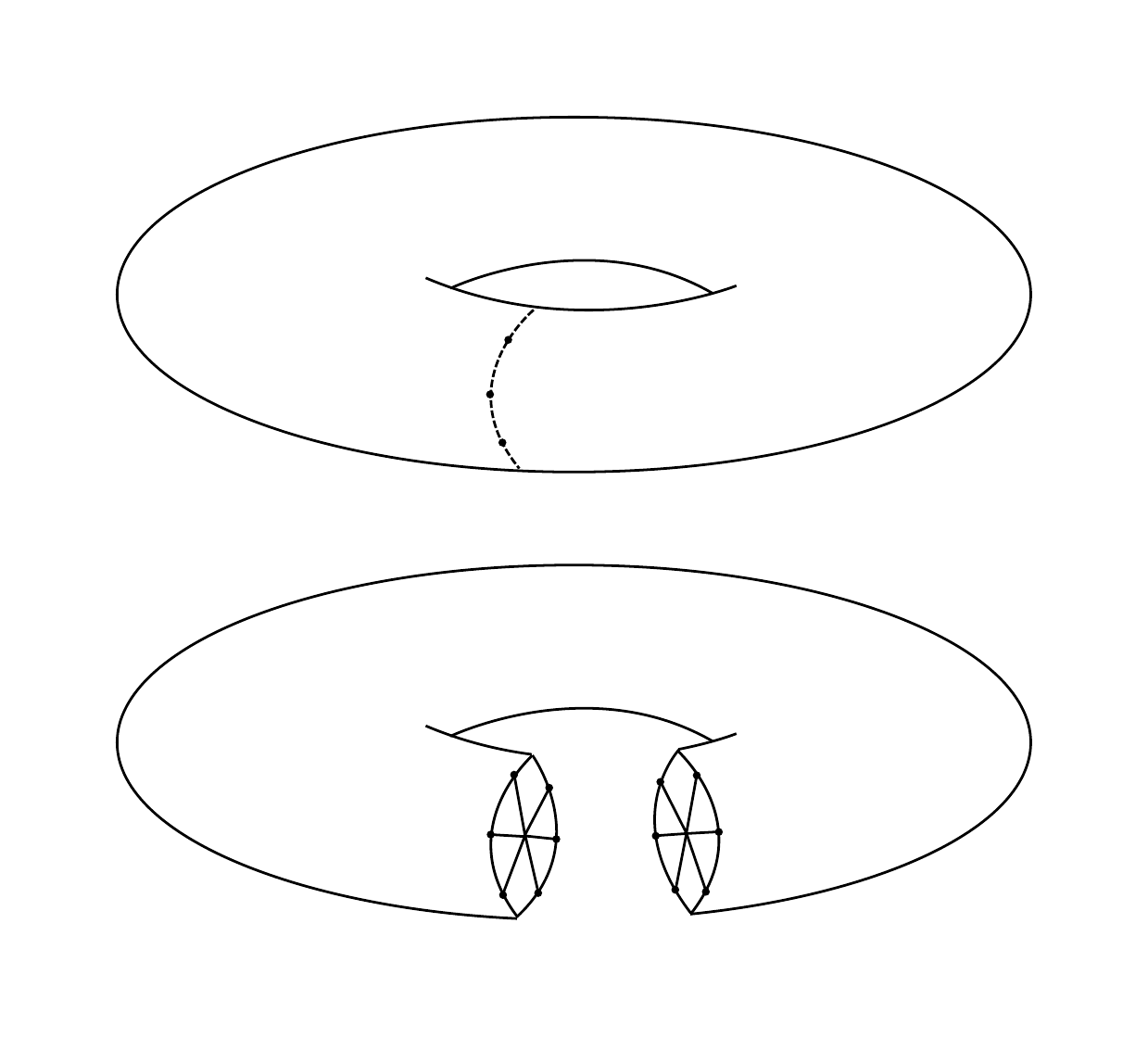}
\caption{An example of the cut-and-cone procedure.}\label{cac}
\end{figure}

\begin{cor}\label{combsys}
  Let $(M,\T)$ be a triangulated surface with infinite fundamental 
  group. Then the systole is bounded by 
  \begin{equation*}
    \sys_{\T}(M) \leq \frac{2}{\sqrt{3}} \kappa_2 \sqrt{\text{Vol}_{\T}(M)}.
  \end{equation*}
\end{cor}

Corollary \ref{combsys} follows from Lemma \ref{sys2comblemma} and Corollary 5.2.B \cite{Gromov1}.

\begin{cor}\label{combhsys}
  Let $(M,\T)$ be a triangulated surface of genus $g > 0$. Then the homological systole is bounded by
  \begin{equation*}
    \sys ^H_{\T}(M) \leq K_g \frac{\log g}{\sqrt g} \sqrt{\text{Vol}_{\T}(M)}.
  \end{equation*}
  where $K_g$ depends only on the genus $g$ and not on $M$ or $\T$.
\end{cor}

Corollary \ref{combhsys} follows from Lemma \ref{sys2comblemma} and Theorem 2.C \cite{Gromov2}.


\subsection*{The cut-and-cone procedure}

Suppose that $(M,\T)$ is a triangulated surface with genus $g \geq 2$.  The $g=0, 1$ cases will be dealt with individually later. In order to simplify notation, we will use $|\T|$ to denote $\text{Vol}_{\T}(M)$, the number of triangles in the triangulation $\T$. Set $(M_{(0)}, \T_{(0)}):= (M, \T)$. By Corollary \ref{combhsys}, there 
exists a homologically nontrivial edge loop $p$ so that
\begin{equation}\label{surfaceloop}
  \ell_{\T}(p) \leq  K \frac{\log g}{\sqrt g}\sqrt{|\T|}.
\end{equation}
By reducing the loop $p$, if necessary, we may assume that $p$ is simple and still satisfies equation (\ref{surfaceloop}).  Cutting $M$ along $p$ yields a connected surface of genus $g-1$ with two boundary components.  We then cone off the two boundary components to obtain a triangulated surface $(M_{(1)}, \T_{(1)})$ with genus $g-1$.  Note that 
\begin{equation*}
  |\T_{(1)}| \leq |\T| + 2 \ell_{\T}(p) \leq |\T| + 2K \frac{\log g}{\sqrt g} \sqrt{|\T|} \leq \left( \sqrt{|T|} + K \frac{\log g}{\sqrt g} \right)^2.
\end{equation*}

Suppose, inductively, that we have triangulated surfaces\footnote{Note that, in what follows, we are abusing notation and using $\T_{(i)}$ to denote both the surface and the triangulation.}
\begin{equation*}
  \T = \T_{(0)}, \T_{(1)},\dots,\T_{(n)}
\end{equation*}
where $n \leq g - 1$, $\T_{(i)}$ is obtained from $\T_{(i-1)}$ by the above cut-and-cone procedure, and we have
\begin{equation*}
  |\T_{(i)}| \leq \left( \sqrt{|\T|} +  K\sum_{k=g-(i-1)}^g\frac{\log k}{\sqrt k} \right)^2.
\end{equation*}

If $n < g-1$, then $\T_{(n)}$ has genus $g - n \geq 2$, so by Theorem
\ref{combhsys}, there exists a homologically nontrivial edge loop $p_{(n)}$ so 
that
\begin{equation*}
  \ell_{\T_{(n)}}\left(p_{(n)}\right) \leq  K \frac{\log 
  (g-n)}{\sqrt{g-n}}\sqrt{|\T_{(n)}|}.
\end{equation*}
We may cut $\T_{(n)}$ along this path and cone off the boundaries to get a 
triangulated surface $\T_{(n+1)}$ 
with genus one less than the genus of $\T_{(n)}$ so that
\begin{align*}
  |\T_{(n+1)}| &\leq |\T_{(n)}| + 2K \frac{\log (g-n)}{\sqrt{g-n}}\sqrt{|\T_{(n)}|}\\
                 &\leq \left( \sqrt{|\T|} +  K\sum_{k=g-(n-1)}^g\frac{\log k}{\sqrt k} \right)^2 \\
                 &+ 2K\frac{\log (g-n)}{\sqrt{g-n}}\left( \sqrt{|\T|} +  K\sum_{k=g-(n-1)}^g\frac{\log k}{\sqrt k} \right)\\
                 &\leq \left( \sqrt{|\T|} +  K\sum_{k=g-n}^g\frac{\log k}{\sqrt k} \right)^2.
\end{align*}

If $n = g - 1$, then $\T_{(n)} = \T_{(g-1)}$ is a torus and we may 
apply Corollary \ref{combsys} to get a noncontractible edge loop $p$ so that
\begin{equation}\label{torusloop}
  \ell_{\T_{g-1}}(p) \leq \frac{2}{\sqrt{3}} \kappa_2 \sqrt{|\T_{(g-1)}|}.
\end{equation}
Cutting and coning along $p$ gives us a triangulated 2-sphere
$\T_{(g)}$ such that
\begin{align*}
  |\T_{(g)}| &= |\T_{(g-1)}| + 2\ell_{\T_{g-1}}(p)\\
               &\leq |\T_{(g-1)}| + 2 \left( \frac{2}{\sqrt{3}} \kappa_2 \right) \sqrt{|\T_{(g-1)}|}\\
               &\leq \left( \sqrt{|\T|} +  K\sum_{k=2}^g\frac{\log k}{\sqrt k} \right)^2 + 
 \frac{4}{\sqrt{3}} \kappa_2 \left( \sqrt{|\T|} +  K\sum_{k=2}^g\frac{\log k}{\sqrt k} 
  \right)\\
  &\leq 11\left( \sqrt{|\T|} +  K\sum_{k=2}^g\frac{\log k}{\sqrt k} \right)^2.
\end{align*}

If $n=g$, then $\T_{(n)} = \T_{(g)}$ is a $2$-sphere.  We need to perform a special coning off of $\T_{(g)}$.  The reason for this is to ensure that, when we glue the surface back to gether to get our $3$-dimensional filling $(N,\T_N)$, we obtain a legitimate simplicial complex decomposition for $N$.  If we would just cone off $\T_{(g)}$, then various tetrahedra could intersect at both the cone point and in their opposite face.  

The procedure for the modified coning of the $2$-sphere is as follows.  
For each simplex $\sigma$ of $\T_{(g)}$, we will triangulate the prism $\sigma \times I$, where $I$ is the unit interval, in the same manner as used by Hatcher in \cite{Hatcher}. 
Suppose the vertices of $\sigma \times \{1\}$ are $\{v_0, v_1, v_2\}$, where the indices represent some fixed ordering of the vertices. 
Let the corresponding vertices of $\sigma \times \{0\}$ be $\{w_0, w_1, w_2\}$.
Then the simplices $\langle v_0v_1v_2w_2 \rangle$, $ \langle v_0v_1w_1w_2 \rangle$, and $ \langle v_0w_0w_1w_2 \rangle$ triangulate $\sigma \times I$, and if we do this for each simplex of $\T_{(g)}$, adjacent simplices will have consistent triangulations.
Finally we cone off $\T_{(g)} \times \{0\}$ to get a triangulated 3-ball $B_3$ which has two layers: the center, which is a coned off copy of $\T_{(g)}$, and the exterior shell, which is our triangulated $\T_{(g)} \times I$.  Note that
\begin{equation*}
  |B_3| = 4|\T_{(g)}| \leq 44\left(\sqrt{|\T|} +  K\sum_{k=2}^g\frac{\log k}{\sqrt k} \right)^2.
\end{equation*}

By gluing together $B_3$ along the cuts in the reverse order, we obtain a 
triangulated 3-manifold $(N,\T^\prime)$ which is a filling of $(M,\T)$ and
\begin{align*}
  |\T^\prime| &\leq 44\left( \sqrt{|\T|} +  K\sum_{k=2}^g\frac{\log k}{\sqrt k} \right)^2\\
      &= 44\left( \sqrt{|\T|} +  K\sum_{k=2}^7\frac{\log k}{\sqrt k} + K\sum_{k=8}^g\frac{\log k}{\sqrt k} \right)^2\\
      &\leq 44\left( \sqrt{|\T|}+ C'+\int_7^g \frac{\log x}{\sqrt x}\, dx\right)^2\\
      &\leq 44\left( \sqrt{|\T|}+ C'+2\sqrt g \log g\right)^2.
\end{align*}

\subsection{Proofs of theorems}

\begin{proof}[Proof of Theorem \ref{lineargenus}]
  Suppose $(M,\T_M)$ is a triangulated surface of genus at most $g$. After 
  performing the above cut-and-cone procedure we obtain $(N,\T_N)$, a filling of $M$, so that
  \begin{align*}
    |\T_N| &\leq 44\left( \sqrt{|\T_M|}+ C'+2\sqrt g \log g\right)^2\\
        &\leq 44\left(\sqrt{|\T_M|}+ C_g'\right)^2\\
        &\leq C_g|\T_M|
  \end{align*}
  for a suitable constant $C_g$.
\end{proof}

\begin{proof}[Proof of Theroem \ref{quadratic}]
  Suppose $(M,\T_M)$ is a triangulated surface of genus $g$.  After 
  performing the above cut-and-cone procedure we obtain $(N,\T_N)$, a filling of $M$, so that
  \begin{equation}\label{notris}
    |\T_N| \leq 44\left(\sqrt{|\T_M|}+ C'+2\sqrt g \log g \right)^2.
  \end{equation}

  Since $M$ is closed, the number of edges in $\T_M$ is $(3/2)|\T_M|$.
  Thus if $|v(\T_M)|$ is the number of vertices of $\T_M$, we know that the Euler characteristic $\chi(\T_M)$ satisfies
  \begin{equation*}
    2 - 2g = \chi(\T_M) = |v(\T_M)| - \frac{|\T_M|}{2}.
  \end{equation*}
  Solving for $g$ then gives that
  \begin{equation}\label{euler}
   g =  \frac{-|v(\T_M)|}{2} + \frac{|\T_M|}{4} + 1 \leq \frac{|\T_M|}{4}.
  \end{equation}
  Combining \eqref{notris} and \eqref{euler}, we can conclude that
  \begin{align*}
    |\T_N| &\leq 44\left( \sqrt{|\T_M|}+ C'+2\sqrt g \log g\right)^2\\
        &\leq 44\left( \sqrt{|\T_M|}\left(1 + \log |\T_M|\right) + C'' \right)^2\\
        &= 44\left(|\T_M|\left(1 + \log |\T_M|\right)^2 + 2C''\sqrt{|\T_M|}(1 + \log |\T_M|) + \left( C'' \right)^2 \right)\\
        &\leq C |\T_M|\left( \log |\T_M| \right)^2
  \end{align*}
  for some suitable $C$.
\end{proof}


\section{Concluding remarks}\label{section:concluding-remarks}

Our results suggest a variety of directions for further work. Firstly, note that throughout our paper we restrict 
ourselves to {\it smooth} triangulations. This restriction appears (and is used) in both implications of our Main Theorem. 
In our Theorem \ref{thm:riemannian-to-discrete}, we make use of Whitney's triangulation process, which produces 
smooth triangulations. In the proof of our Theorem \ref{thm:riemannian-to-discrete}, smoothness of the triangulation
is used to produce nice local coordinates near the various faces. Note however that, even restricting to smooth manifolds,
one can find many non-smooth triangulations. Indeed, for a smooth triangulation, it follows that the image of every simplex 
is smoothly embedded, and hence has link homeomorphic to a sphere of the appropriate codimension. On the other hand, 
the celebrated Cannon-Edwards double suspension theorem (see \cite{Ed} and \cite{C})
states that, if one starts with an arbitrary $n$-dimensional 
homology sphere $H$ (i.e. a connected $n$-manifold whose integral homology vanishes in all degrees $\neq 0,n$), the double suspension $\Sigma^2 H = H* S^1$ is homeomorphic to $S^{n+2}$. Triangulating both $H$ and the $S^1$, we get an 
induced triangulation of the join $\Sigma^2 H = H*S^1$, and hence a triangulation of $S^{n+2}$. But in this triangulation,
the edges in the $S^1$ have links homeomorphic to $H$. As a result, any homeomorphism $\Sigma^2H \rightarrow S^{n+2}$
must take the suspension curve $S^1$ to a non-smooth curve in $S^5$. This yields a triangulation of $S^{n+2}$ which is 
not smooth (in fact, not even PL). 

\vskip 10pt

\noindent {\bf Question:} If we have a class of smooth manifolds for which the systolic inequality holds for all smooth triangulations,
does the systolic inequality still hold (possibly with a different constant) for {\it all} triangulations? How about for PL-triangulations?

\vskip 10pt

For a Riemannian manifold, we construct smooth triangulations whose simplices are ``metrically nice'' (as seen
by the Riemannian metric). In the realm of Riemannian geometry, perhaps the most important notion is that of curvature.
It is reasonable to ask whether one can produce smooth triangulations which also respect the curvature of the 
underlying metric.

\vskip 10pt

\noindent {\bf Question:} If $M$ is a closed negatively curved manifold, does $M$ support a piecewise Euclidean, 
locally CAT(0) metric?

\vskip 10pt

The CAT(0) condition is a metric version of non-positive curvature. In the special case where $M$ is (real) hyperbolic,
this question has an affirmative answer, by work of Charney, Davis, and Moussong \cite{CDM}. Note that, if one replaces
the ``negatively curved'' by ``non-positively curved'', then there are counterexamples (due to Davis, Okun, and Zheng
\cite{DOZ}).

\vskip 5pt

In Corollary \ref{cor:essential} we used the direction $(1) \Rightarrow (2)$ of our Main Theorem to prove that a specific class of triangulated manifolds satisfied the combinatorial systolic inequality.
So the following question is very natural:

\vskip 10pt

\noindent {\bf Question:} Can one directly establish the combinatorial systolic inequality for some classes of manifolds?

\vskip 10pt

Via the implication $(2) \Rightarrow (1)$ in the Main Theorem, this would imply corresponding 
Riemannian systolic inequalities.  Finally, we can ask for improvements on the filling function for triangulated surfaces:

\vskip 10pt

\noindent {\bf Question:} Does the filling function for triangulated surfaces satisfy a linear bound, with constant independent
of the genus?

\vskip 10pt

In our Theorem \ref{lineargenus}, we showed that for each fixed genus $g$, one has a linear filling function (but with 
a constant that depends on the genus). If we try to get a genus independent estimate, our Theorem \ref{quadratic} 
gives a slightly worse bound, with an additional log squared factor. It is unclear whether or not we should expect an
affirmative answer to the last question. Of course, the question of finding a good filling is also of interest in higher 
dimension.

\vskip 10pt

\noindent {\bf Question:} If $M$ is a manifold which bounds, what can one say about the filling function for $M$? For instance, for closed $3$-manifolds, can one compute the (optimal) filling function? Could these filling functions be used to distinguish the topology of the $3$-manifold?

\vskip 10pt

If one just tries to minimize the numbers of simplices over all possible triangulations, then Costantino and D. Thurston \cite{CT}
have some estimates on the corresponding filling function (note that they do not require the optimal triangulation on the $4$-manifold to restrict to the given triangulation on the $3$-manifold).

\section*{Appendix:  Results of Whitney from \cite{Whitney} used in Section \ref{section:whitneyarg}}\label{section:appendix}
\setcounter{equation}{0}
 \setcounter{thm}{0}
 \numberwithin{thm}{section}
 \renewcommand{\thesection}{A}
Here we present many facts used in the proof of Whitney's triangulation theorem in Section \ref{section:whitneyarg}. 
 Detailed arguments can be found in Whitney's book \cite{Whitney}.  A few footnotes have been added to help the reader recall definitions from Sections \ref{section:background} and \ref{section:whitneyarg}.

  We first discuss some properties of the fullness of a simplex.

  The first property is that fullness of a simplex implies fullness of all of its faces.
  That is, if $\sigma^k$ is a face of the simplex $\sigma^r$, then
  \begin{equation}\label{14.6}
    r! \Theta(\sigma^r) \leq k!\Theta(\sigma^k).
  \end{equation}
  
   Fullness is also nearly preserved if the vertices of the simplex are not moved too much.
  \begin{lem}[IV, 14c \cite{Whitney}]\label{14c}
    Given $r \in \mathbb{N}$, $\Theta_0 > 0$, and $\epsilon > 0$, there is a $\rho_0 > 0$ with 
    the following property. Take any simplex $\sigma = <p_0, \cdots, p_r>$ with 
    $\Theta(\sigma) \geq \Theta_0$, and take points $q_0, \dots, q_r$ with $|q_i 
    - p_i | \leq \rho_0 \diam(\sigma)$. Then the simplex $\sigma' = \langle q_0, \cdots, q_r \rangle $ satisfies
     $\Theta(\sigma') \geq \Theta_0 - \epsilon$.
  \end{lem}
  
  \begin{lem}[IV, 14b \cite{Whitney}]\label{14b}
  For any $r$-simplex $\sigma = \langle p_0,\cdots, p_r \rangle$ and point $p = \mu_0p_0+ \cdots \mu_rp_r \in \sigma$,
  \begin{equation}\label{14.7}
    \dist(p, \partial \sigma) \geq r! \Theta(\sigma) \diam(\sigma) \inf\{\mu_0,\dots,\mu_r\}.
  \end{equation}
\end{lem}

 \begin{lem}[IV, 15c \cite{Whitney}]\label{15c}
    Let $\pi$ be the orthogonal projection onto a subspace $P$. 
    Let $\sigma = \langle p_0,\cdots, p_r \rangle$ be a simplex, and suppose that\footnote{$U_{\zeta}(P)$ is just the $\zeta$ neighborhood of $P$.} $\sigma \subset U_\zeta(P)$ and $|p_i - p_0| \geq \delta > 0$ for all $i > 0$.
    Then for any unit vector $u \in \sigma$,
    \begin{equation*}
      |u - \pi (u)| \leq \frac{2\zeta}{(r-1)!\Theta(\sigma)\delta}.
    \end{equation*}
  \end{lem}

Now we list some results used to pick various quantities in the proof.

\begin{lem}[App. II, 16a \cite{Whitney}]\label{aII,16a}
  Given the integer $m$, there is a number $\rho^* > 0$ with the following 
  property. Let $K_0$ be a subdivision of $\R^m$ into cubes of side length $h$, 
  and let $K$ be the barycentric subdivision of $K_0$, with vertices $\{ p_i \}$. For each $i$, let $p_i'$ be a point with
  \begin{equation}\label{aII.1}
    |p_i' - p_i| \leq \rho^* h.
  \end{equation}
  Let $f$ be the affine mapping of $K$ into $\R^m$ defined by 
  $f(p_i) = p_i'$. Then $f$ is an injective map from $K$ onto $\R^m$, and the simplices
  $f(\sigma)$ form a simplicial subdivision of $\R^m$.
\end{lem}

\begin{lem}\label{rho1}
  Let $N$ be a natural number. Then there is a $\rho_1$ with the following 
  property: For any ball $B$ in $\R^m$ of any radius $a$, let $B'$ be the part 
  of $B$ between any two parallel $(m-1)$-planes whose distance apart is less 
  than $2 \rho_1 a$ apart. Then we have that
  \begin{equation}\label{17.1}
    \text{Vol}(B') < \frac{\text{Vol}(B)}{N}.
  \end{equation}
\end{lem}

\begin{proof}
  Let $\rho_1$ be such that
  \begin{equation*}
    0 <\rho_1 < \frac{\text{Vol}(\B^{m})}{2N\text{Vol}(\B^{m-1})},
  \end{equation*}
  where $\B^k$ is the $k$-dimensional unit ball. Suppose two $(m-1)$-planes are
  some distance $d < 2\rho_1a$ apart. Let $B'$ be the volume of the region
  between them contained in $B$. We then have that
  \begin{align*}
    \text{Vol}(B') &< d \text{Vol}(\B^{m-1})a^{m-1}\\
            & < 2\rho_1 \text{Vol}(\B^{m-1}) a^m\\
            & < \frac{\text{Vol}(\B^m) a^m}{N}\\
            & = \frac{\text{Vol}(B)}{N}.
  \end{align*}
\end{proof}

\begin{thm}{(IV, 10A \cite{Whitney})}\label{10A}
  Let $M$ be a smooth, compact, $n$-dimensional submanifold of $\R^m$.
  For each $p \in M$, let $P_p^*$ be the $s$-plane in $\R^m$ normal to $M$ (where $s=m-n$).
  Then there exists a positive number $\delta_0 >0$ with the following properties. 

  Set
  \begin{equation}\label{10.2}
    Q^*_p = P_p^* \cap U_{\delta_0}(p_0).
  \end{equation}
  The $Q_p^*$ fill out a neighborhood $U^*$ of $M$ in an injective way. Set
  \begin{equation}\label{10.3}
    \pi^*(q) = p \text{ if } q \in Q^*_p.
  \end{equation}
  This is a smooth mapping of $U^*$ onto $M$, and
  \begin{equation}\label{10.4}
    |\pi^*(q) - q| \leq 2 \dist(q, M), \quad q \in U^*.
  \end{equation}
\end{thm}

For the following two Lemmas, recall the definition of $M_{p,\xi}$ from Section \ref{section:background}.
\begin{lem}[IV, 8a \cite{Whitney}]\label{8a}
  Let $M$ be a compact submanifold of $\R^m$. Then there is a $\xi_0 >0$ such that $M_{p, \xi_0}$
   is defined for all $p \in M$. Moreover,
  \begin{equation}\label{8.3}
    \dist(p, M \setminus M_{p, \xi}) \geq \xi \quad \text{ if } \quad \xi \leq \xi_0.
  \end{equation}
\end{lem}

\begin{lem}[IV, 8b,c \cite{Whitney}]\label{8b}
  Let $M$ be a compact submanifold of $\R^m$.  Then for any $\lambda > 0$ there is a $\xi_1 
  >0$ with the following property. For any $p \in M$ and any vector $v$ tangent 
  to $M_{p, \xi_1}$,
  \begin{equation}\label{8.4}
    |v - \pi_p (v)| \leq \lambda |\pi_p(v)| \leq \lambda |v|.
  \end{equation}
  Moreover, we have the above inequality for any secant vector $v$ to $M_{p, 
  \xi_1}$ and we get the following results
  \begin{equation}\label{8.5}
    |p' - \pi_p( p')| < \lambda \xi, \quad \text{for } p' \in M_{p, \xi} \text{ and } \xi \leq \xi_1,
  \end{equation}
  \begin{equation}\label{8.6}
    M_{p,\xi} \subset U_{\lambda \xi}(P_{p,\xi}) \text{ and } P_{p, \xi} \subset U_{\lambda \xi}(M_{p,\xi}),\quad \text{for } \xi \leq \xi_1.
  \end{equation} 
\end{lem}

We now present some facts about affine subspaces of $\R^m$.

\begin{lem}[App. II, 14b \cite{Whitney}]\label{aII,14b}
  Let $P^*$ be an affine subspace of $\R^m$, let $P$ be an affine subspace of $P^*$, let $Q$ be a closed set in $P^*$, let $p^*$ be a point of $\R^m$ not in $Q$, and let $Q^*$ be the join of $p^*$ and $Q$.
  Then
  \begin{equation*}
    \dist(Q^*, P) \geq \frac{\dist(Q,P) \cdot \dist (p^*, P^*)}{\diam(Q^*)}.
  \end{equation*}
\end{lem}

\begin{lem}[IV, 10a \cite{Whitney}]\label{10a}
  Take $\lambda < 1$ and $\xi_1 \leq \xi_0$ as in Lemmas \ref{8a} and \ref{8b}.
  Take any $p,p' \in M$ with $|p - p'| < \xi_1$.
  Then\footnote{Recall that $P_p^*$ denotes the normal plane to $M$ at $p$.} $P_p^*$ intersects $P_p$ in a unique point, and if $v \in P_{p'}^*$, then
  \begin{equation}\label{10.12}
    |\pi_p v| \leq \lambda |v|.
  \end{equation}
\end{lem}



\end{document}

%% file: Ppxi.tex
\begin{tikzpicture}[scale=1.5]
  \coordinate (p1) at (-4,-0.5);
  \coordinate (p2) at (2,-0.5);
  \coordinate (p3) at (4,0.5);
  \coordinate (p4) at (-2,0.5);
  \coordinate (labelPp) at ($(p3)!0.8!(p4)$);
  \coordinate (p) at (0,0);
  \coordinate (m1) at (-3,-1);
  \coordinate (m2) at (3,-1);
  \draw [white, name path=mxi, thin] (p) ellipse (1.6cm and 0.46cm);
  \draw [help lines] (p1) -- (p2) -- (p3) -- (p4) -- cycle;
  \node [above] at ($(p3)!0.2!(p4)$) {$P_p$};
  \draw [thick, name path=m1] (m1) to [out=30, in=180] (p);
  \draw [thick, name path=m2] (m2) to [out=150, in=0] (p) ;
  \node [below right] at (m2) {$M$};
  \draw [thin, name intersections={of=m1 and mxi}]
  (intersection-1) to [out=-160, in=180] ([yshift=-0.5cm]p);
  \draw [thin, name intersections={of=m2 and mxi}]
  (intersection-1) to [out=-20, in=0] ([yshift=-0.5cm]p);
  \node [below] at ([xshift=-0.6cm,yshift=-0.5cm]p) {$M_{p,\xi}$};

  \draw [dashed, fill=white] (p) ellipse (1.5cm and 0.25cm);
  \node [above right] at ([xshift=1.4cm]p) {$P_{p,\xi}$};
  \draw [fill] (p) circle [radius=0.03];
  \node [right] at (p) {$p$};
\end{tikzpicture}

%% file: complexK.tex
\begin{tikzpicture}[scale=1.5]
  \coordinate (a1) at (-3,1);
  \coordinate (a2) at (3,1);
  \coordinate (a3) at (0,-1.5);
  \coordinate (a12) at ($(a1)!0.75!(a2)$);
  \coordinate (b1) at ($(a1)!0.3!(a3)$);
  \coordinate (b2) at ($(a2)!0.3!(a3)$);
  \coordinate (b12) at ($(b1)!0.5!(b2)$);
  \coordinate (labelsigma0) at ($(a1)!0.7!(a3)$);
  \coordinate (labelsigma1) at ($(a2)!0.7!(a3)$);
  \coordinate (labelsigma2) at ($(b12)!0.5!(a3)$);
  \coordinate (labelK) at ($(b1)!0.25!(b2)$);
  \draw [dashed] (a1) -- (a2) -- (a3) -- cycle;
  \draw [fill] (a1) circle [radius=0.03];
  \draw [fill] (a2) circle [radius=0.03];
  \draw [fill] (a3) circle [radius=0.03];
  \node [below left] at (labelsigma0) {$\sigma^1_0$};
  \node [below right] at (labelsigma1) {$\sigma^1_1$};
  \node  at (labelsigma2) {$\sigma^2$};
  \node [above] at (a12) {$L^*$};

  \node [fill=white, above] at (b1) {$\psi(\sigma^1_0)$};
  \node [fill=white, above] at (b2) {$\psi(\sigma^1_1)$};
  \node [below] at (b12) {$\psi(\sigma^2)$};
  \node [below] at (labelK) {$K$};
  \draw [thick](b1) -- (b2);
  \draw [fill] (b1) circle [radius=0.03];
  \draw [fill] (b2) circle [radius=0.03];
  \draw [fill] (b12) circle [radius=0.03];
  \draw [thin] (b1) to [out=20, in=160] (b2);
  \draw [thin] (b1) to [out=200, in=40] ([xshift=-0.5cm,yshift=-0.3cm]b1);
  \draw [thin] (b2) to [out=-20, in=140] ([xshift=0.5cm,yshift=-0.3cm]b2);
  \node [right] at ([xshift=0.5cm,yshift=-0.3cm]b2) {$M$};

\end{tikzpicture}